\documentclass[12pt]{article}

\usepackage{geometry, authblk}

\usepackage{enumerate}
\usepackage[OT1]{fontenc}
\usepackage{amsmath,amssymb, algorithm}
\usepackage{natbib}
\usepackage[usenames]{color}
\usepackage{smile}
\usepackage[protrusion=true, expansion=true, tracking=true, spacing=true, final, babel]{microtype}
\usepackage{natbib}
\usepackage{multirow}
\usepackage[colorlinks,
linkcolor=red,
anchorcolor=red,
citecolor=red
]{hyperref}

\def\tr{\mathop{\text{tr}}\kern.2ex}

\def\E{{\mathbb E}}

\long\def\comment#1{}

\def\cS{{\mathcal{S}}}

\newcommand{\la}{\langle}
\newcommand{\ra}{\rangle}

\usepackage{mathrsfs}

\usepackage{hyperref}

\def\##1\#{\begin{align}#1\end{align}}
\def\$#1\${\begin{align*}#1\end{align*}}
\usepackage{enumitem}

\begin{document}








\title{Tensor Methods for Additive Index Models under Discordance and Heterogeneity\footnote{Authors listed in alphabetical order.   Krishnakumar Balasubramanian and Jianqing Fan are supported by the DMS-1662139 and DMS-1712591 and NIH grant 2R01-GM072611-13.}}
\author[1]{Krishnakumar Balasubramanian\thanks{kb18@princeton.edu}}
\author[1]{Jianqing Fan\thanks{jqfan@princeton.edu}}
\author[1]{Zhuoran Yang\thanks{zy6@princeton.edu}}
\affil[1]{Department of Operations Research and Financial Engineering, Princeton University}


\maketitle

\begin{abstract}
Motivated by the sampling problems and heterogeneity issues common in high-dimensional big datasets, we consider a class of discordant additive index models. We propose method of moments based procedures for estimating the indices of such discordant additive index models in both low and high-dimensional settings. Our estimators are based on factorizing certain moment tensors and are also applicable in the overcomplete setting, where the number of indices is more than the dimensionality of the datasets. Furthermore, we provide rates of convergence of our estimator in both high and low-dimensional setting. Establishing such results requires deriving tensor operator norm concentration inequalities that might be of independent interest. Finally, we provide simulation results supporting our theory. Our contributions extend the applicability of tensor methods for novel models in addition to making progress on understanding theoretical properties of such tensor methods.
\end{abstract}



\newpage

\section{Introduction} \label{sec:intro}

High-dimensional big datasets are typically collected by aggregating information from a variety of sources. Such diverse sources of data, invariably cause sampling challenges and heterogeneity issues~\citep{fan2014challenges}. Motivated by these concerns, in this work, we consider a class of additive index models (AIMs) described as follows. Given the covariate $X$,  we let
$
\mathcal{Z} = \{ Z_j=f_j( \langle X, \beta^*_j \rangle ) + \epsilon_j  \}_{j \in [k]} $
be an underlying set of latent models. 
That is, $\cZ$ is an \emph{unordered} collection of the responses of $k$ single index models (SIMs), that are unobservable. Here $\{ \epsilon_j  \}_{j \in [k]}$ are exogenous noise.  Moreover, based on $\mathcal{Z}$, the observed response $Y$ is given by  a real-valued function $h: \mathcal{Z} \to \mathbb{R}$ 
such that
\begin{align}\label{eq:discoaim}
\E (Y|X) = \E (h(\mathcal{Z})|X) =  g_1( \langle X, \beta^*_1 \rangle ) +  g_2( \langle X, \beta^*_2 \rangle ) + \ldots +   g_k ( \langle X, \beta^*_k \rangle ).
\end{align}
Here, the functions $g$ actually depend on the choice of $h$ and $f$. To avoid notational clutter, we just use $g$ to denote it. In addition, we note that the function $h$ can be a random function that is independent of $\cZ$. We call this model as discordant additive index model (DAIM), as the response $Y$ depends on the latent set $\cZ$.  We allow the model to be overcomplete, where the number of index models, $k$, is larger than the dimensionality $d$.
 Given $n$ i.i.d. observations $\{ (X^{(i)}, h(\mathcal{Z}^{(i)}))\}_{i \in [n]}$ of the DAIM in \eqref{eq:discoaim}, our goal is to  estimate the parametric components $\{ \beta^*_j\}_{j \in [k]}$.

Our estimators are based on method of moments approach and involve factorizing higher-order moment tensors. The corresponding second-order analogue, called as principal component analysis, has been leveraged widely for estimation in several statistical problems, for example, factor modeling, dimensionality reduction, estimation in mixture models and community detection. Furthermore, inferential and computational properties of such estimators for the above problems are relatively well understood. It is worth noting that establishing operator norm bounds for certain random matrices played a crucial part in deriving such results. We refer the reader to~\cite{fan2018principal} for a detailed survey of such results. In comparison, for the higher-order case, both the methodology and the theory of tensor factorization approaches are still in its infancy. In this paper, we adopt the tensor decomposition framework \citep{anandkumar2014tensor, anandkumar2014guaranteed, sun2017provable} for estimating the parameters of the DAIM in \eqref{eq:discoaim}, and also derive operator norm bounds for the decomposed moment tensors, thereby making progress towards understanding theoretical properties of higher-order decompositions and widening their applicability.

We now provide two concrete instantiations of the DAIM by appropriately defining the sampling function $h$ and motivate their applications. A canonical example of DAIM are the mixture models used to model heterogeneity~\citep{mclachlan2004finite, fan2014challenges, stadler2010l1}. In this setting, only one element of the latent $\mathcal{Z}$ is observed. That is, $h(\mathcal{Z})$ is assumed to be a random function that picks element $Z_j$ with probability $\pi_j$ where $\sum_{j \in [k]} \pi_j =1$. Hence the observation is given by $Y = h(\mathcal{Z}) = Z_j$ with probability $\pi_j$. Again, it is easy to see that the model is an instantiation of the DAIM posited in~\eqref{eq:discoaim}, with appropriately defined $g$ functions. Correspondingly in the sample setting, given $n$ i.i.d samples $X^{(i)} \in \mathbb{R}^d$ and the matrix $Z_{ij} \in \mathbb{R}^{k \times n}$, for each column, one of the $k$ entries is observed and the probability of observing the $j$-th entry is given by $\pi_j$ independently for all columns. From a practical perspective, heterogenous data is ubiquitous. For example, in genomics and neuroscience, the occurrences of systematic biases is natural due to the data being combined from multiple sources. Similarly in financial econometrics, the sources of data available includes stocks, trading data and unstructured text from news and blog sources. In these situations, failure to acknowledge data heterogeneity leads to wrong inferences~\citep{fan2014challenges}. 
Furthermore, in practice the non-parametric component can be misspeficied and the number of components, $k$, can be greater than the dimensionality of the dataset causing more challenges for estimation.



Yet another example of the DAIM is the problem of correspondence retrieval proposed recently in~\cite{andoni2017correspondence}. Here, the correspondence between the model under consideration ($\beta_j$) and the actual responses at hand ($Z_j$) is unobserved. Instead, we observe say the average (or sum) of the responses, $ k^{-1} \sum_{j \in [k]} Z_j$. It is clear that if the sampling function is taken as $h (\mathcal{Z}) = k^{-1} \sum_{j \in [k]} Z_j= k^{-1}\sum_{j \in [k]} [f_j( \langle X, \beta^*_j \rangle ) + \epsilon_j ] $, the model is an instantiation of the DAIM posited in~\eqref{eq:discoaim}, with appropriately defined $g$ functions that depend on the corresponding $f$ functions. In the sample setting, given $n$ i.i.d.  samples $X^{(i)} \in \mathbb{R}^d$, consider the matrix $Z \in \mathbb{R}^{k \times n}$ where each $Z_{ij} = f_j(\langle X^{(i)}, \beta_j\rangle) + \epsilon_j^{(i)}$. {\color{black} Let $\{ \cZ^{(i)}\}_{i\in [n]}$ be $n$ i.i.d. sets of latent observations, where $\cZ^{(i)} = \{ Z_{ij} , j\in [k]\}$.} Then the $i^{th}$ responses $Y_i$ is given by $h(\mathcal{Z}_i) = k^{-1}\sum_{j \in [k]}  Z_{ij}$. That is, we do not observe the matrix $Z$ which has the correspondence information in the sample setting.  Instead, we observe discordant observations of the form $h(\mathcal{Z}^{(i)})$, which obscures the correspondence information. From a practical point of view, such a lack of correspondence occurs in several situations. For example in high-dimensional nonlinear multi-task learning models~\citep{yang2009heterogeneous}, due to privacy or record linkage issues, one might not observe the correspondence between the response and the covariates of the different models. Furthermore, our models also are applicable to nonlinear compressed sensing in the discordant setting which has wide applications as described in~\citep{unnikrishnan2015unlabeled}. Similar to the previous example, model misspecification and having a large number of components, $k$, cause significant challenges for estimation.

\vspace{0.1in}
\noindent \textbf{Contributions:} Motivated by the above discussion, we focus on the task of estimating the parametric components of the DAIM, while being agnostic to the nonparametric components. First, it is worth noting that likelihood/least-squares based approaches depend on the specification of the nonparametric components and is not agnostic to them. Hence, we propose to use a higher-order moment decomposition based procedure to provably recover the parametric components with unknown nonparametric components. Under a Gaussian design assumption on the covariate and under mild regularity assumptions on the unknown nonparametric components (the specific details  are provided in \S\ref{sec:methods}), we provide polynomial-time computable estimators that achieve optimal statistical rates, in both low and high-dimensional settings, with sufficiently large number of samples. The statistical rates for our estimator are established based on obtaining novel concentration inequalities in tensor operator norm for the considered moment tensors, which might be of independent interest. Finally, from a practical point of view, our estimator for the parametric components could also be used as initializers for alternating minimization algorithms (e.g., the EM algorithm) to estimate the nonparametric components efficiently.



\vspace{0.1in}
\noindent \textbf{Notations:} We denote by $[n]$ the set of integers  $\{ 1, \ldots, n \}$. Furthermore, for a vector $u \in \mathbb{R}^d$ and an index set $F \subseteq [d]$, the truncation of $u$ with respect to the set $F$, denoted as $\vartheta_F (u)$, is defined coordinate-wise as
\begin{align*}
\left[\vartheta_F(u) \right]_i = \begin{cases} u_i & \text{if}~ i \in F\\
0 & \text{otherwise.}
\end{cases}
\end{align*}
We also use the notation $\vartheta_s(u)$ to denote the case where  $F$ consists of   the top-$s$   entries of  $u$ in absolute value. We denote an $\ell$-th order symmetric tensor by $A \in \mathbb{R}^{d\otimes \ell}$. Recall that a tensor is symmetric if $A_{j_1 j_2 \cdots j_\ell}= A_{j_{\sigma(1)} j_{\sigma(2)} \cdots j_{\sigma(\ell)}}$, for every permutation $\sigma$ of the indices. For a given vector $u \in \mathbb{R}^d$, we define the $\ell$-th order  rank-1 tensor formed from $u$ as $ u^{\otimes \ell}$. Similarly, for a given set of $k$ vectors $u^{(1)}, \ldots, u^{(\ell)} \in \mathbb{R}^d$, the rank-1 tensor formed by taking the outer product of them is given by $u^{(1)} \otimes u^{(2)} \otimes \cdots \otimes u^{(\ell)}$. In addition, let $A$ and $B$ be two $\ell$-th order tensors,  we define the inner product between $A$ and $B$ as
$
\la A, B \ra  =  \sum_{j_1=1}^d \cdots \sum_{j_\ell = 1}^d A_{j_1j_2\ldots j_\ell} \cdot B _{j_1 j_2 \ldots j_k }$. Furthermore, for any   $A \in \mathbb{R}^{d \otimes \ell}$ and  $p \in (0, \infty) $, we define the element-wise $\ell_p$-norm as $\|A\|_p = (\sum_{j_1, \ldots, j_\ell=1}^d | A_{j_1 j_2 \cdots j_\ell} |^p)^{1/p}$. Note that this generalizes the standard element-wise $\ell_p$-norm of a vector. 
We also denote the $\ell$-th order polynomial form of the tensor $A$ as $A(u^{(1)} , \ldots, u^{(\ell)}) = \bigl  \la A ,   u^{(1)} \otimes u^{(2)} \otimes \cdots \otimes u^{(\ell)} \bigr \ra$. Note that here $A(u^{(1)} , \ldots, u^{(\ell)})$ is a function of $\{ u^{(j) }\}_{j \in [\ell]}$. The operator norm of a tensor $A$ is then defined as
\#\label{eq:def_tensor_opernorm}
\| A \|_{\oper} = \sup\bigl \{  \bigl | A(u^{(1)} , \ldots, u^{(\ell)} ) \bigr | \colon  u^{(1)}, u^{(2)}, \ldots, u^{(\ell)} \in \mathbb{S}^{d-1} \bigr \},
\#
where $\mathbb{S}^{d-1} = \{ u \in \RR^d \colon \| u \|_2 = 1 \}$ is the unit sphere in $\RR^d$. Note that when the tensor is symmetric, the operator norm can alternatively be calculated as
$\| A \|_{\oper} = \sup\bigl \{  \bigl | A(u , \ldots, u ) \bigr | \colon  u \in \mathbb{S}^{d-1} \bigr \}$. Note that when $\ell=2$, we recover the matrix operator norm. We also define the sparse symmetric tensor operator norm, for some $1\leq r \leq d$, as
\#\label{eq:def_sparse_tensor_opernorm}
\| A \|_{\oper,r} = \sup\bigl \{  \bigl | A(u , \ldots, u ) \bigr | \colon  u \in \mathbb{S}^{d-1} ~\text{and}~ \| u \|_0 \leq r\bigr \}.
\#
For a $(\ell-1)$-th order symmetric tensor-valued function $A_{\ell-1}(a): \mathbb{R}^d \mapsto \mathbb{R}^{d \otimes (\ell-1)}$, its derivative is given by an $\ell$-th order symmetric tensor $\nabla_a A_{\ell-1}(a) $ which is defined entry wise as $[\nabla_a A_{\ell-1}(a)]_{i_1,\ldots,i_{\ell}} = \partial [A_{\ell-1}(a)]_{i_1,\ldots,i_{(\ell-1)}} / \partial a_{i_\ell}$. Finally, for a function $f(a) : \mathbb{R}^d \to \mathbb{R}$, its $\ell$-th derivative is denoted as $f^{(\ell)} (a) $. We end this section by describing the format of the paper. In \S\ref{sec:models}, we precisely define the class of DAIM that we consider and outline our estimator which is based on decomposing higher-order moment tensors. In \S\ref{sec:theory}, we present our main results regarding the rate of convergence of our estimators in both low and high-dimensional setting, that involve obtaining tensor operator norm concentration results. The proofs are relegated to \S\ref{sec:aux}.







\section{Model Definition and Estimation}\label{sec:models}
We now introduce the precise definitions of the models that we consider in this work and outline our  moment-based estimation procedure. As discussed in \S\ref{sec:intro}, our primary motivation for the DAIM in \eqref{eq:discoaim} is based on handling discordance and heterogeneity in  the high-dimensional big data settings. In the following, we define  two instantiations of the DAIM and propose the  corresponding estimation procedures.  We first introduce the discordant single index models, which is a special case of  DAIM under discordance.

\begin{definition}[Discordant SIMs] \label{def:sim0}
 We assume that there are $k$ unordered latent single index models denoted by $\mathcal{Z} = \{Z_j = f_j (\la \beta_j^*, X \ra, \epsilon_j ) \}_{j \in [k]}$. Here   $X \in \RR^d$ is the   covariate,  $\{ \epsilon_j \} _{j \in [k]}\subseteq \RR$ are the random noise   independent of $X$, and $\{ f_j \colon\RR^2 \rightarrow \RR \}_{ j \in [k]}$ are unknown link functions. Based on $\cZ$,  we observe the response variable $Y=h_1(\mathcal{Z}) = k^{-1}\sum_{j \in [k]}Z_j$. 


 \end{definition}

In this model, the set of responses $\mathcal{Z}$ is latent and the   response  observed is given by the average function $h_1(\mathcal{Z})$. In addition, since the norm of $\beta^{*}_j$ can be absorbed in the unknown function $f_j$, $\beta_j^*$  is not identifiable. Hence, we assume that $\beta^{*}_j$ has   norm one   for all $j \in [k]$, and focus  on estimating them while being agnostic to the nonparametric components $\{f_j\}_{j\in[k]}$. Furthermore, we assume that the number of latent single index models, $k$, can be larger than the dimensionality $d$, which yields an overcomplete   model. A more precise characterization of  overcompleteness is provided in \S\ref{sec:theory}. Finally, the average function could also be changed to other additive functions for generality. We next define the mixture of single index models, yet another special case of DAIM, to deal with heterogeneity.

  \begin{definition}[Mixture of SIMs] \label{def:sim}
Similar to Definition \ref{def:sim0}, let  $\mathcal{Z} = \{Z_j = f_j (\la \theta_j^*, x \ra, \epsilon_j ) \}_{j \in [k]}$ be the unordered  responses of the $k$ latent single index models.
  In addition, we  assume that  $\tilde{Z} \in [k]$  is a discrete random variable such that  $\PP( \tilde{Z} = j | X = x) = \pi_j $ for any $j \in [k]$. 
  Here we assume that $\sum_{j\in[k]} \pi_j = 1$. Based on $\cZ$ and $\tilde Z$, the  response variable  is $Y=h_2(\mathcal{Z}) = Z_{\tilde{Z}}$.
\end{definition}

 We note that   here $h_2 \colon \cZ \rightarrow \RR$ is a random function.  Similar to the previous case, the components of the mixture model, $\theta_j^*$ are assumed to be normalized and the number of components can be larger than the dimensionality.
 This model can be slightly  generalized  to allow the  mixing proportion  $\PP( Z = j | X = x)$  to be a function of  $\la X, \theta_j^* \ra $. That is, $\PP( Z = j | X = x)= \pi_j(\langle\theta_j^*, x\rangle)$, where   $\pi_j$ is a univariate function. Such a model is more general and has close relationships to the mixture of experts model~\citep{jacobs1991adaptive} and to modal regression~\citep{chen2016nonparametric}. Although we do not concentrate explicitly on this case, our method and the theoretical results could be easily extended to such a general setting.



\subsection{Estimation  via Third-Order Tensor}\label{sec:methods}
We now outline our procedure for estimating the indices ($\{ \beta^* _{j\in [k]} $ and $\{ \theta^*_j\}_{j \in [k]} $ ) for both models. As mentioned in \S\ref{sec:intro}, our estimation procedure is based on decomposing a moment tensor and is \emph{agnostic} to the nonparametric components $\{f_j\}_{j \in [k]}$. Specifically, our method is based on the higher-order score tensors, which are defined as follows.

\begin{definition}[Higher-Order Score Tensors~\citep{janzamin2014score}]
Let $p\colon \RR^d \rightarrow \RR$   be  the probability  density  function of $N(0, I_d)$. For each positive integer $\ell$, we define the $\ell$-th order score function $S_\ell \colon \RR^d \rightarrow \RR^{d \otimes \ell}$ recursively by letting
\#\label{eq:def_score}
S_\ell(x) = - S_{\ell-1}(x) \otimes \nabla \log p(x) + \nabla S_{\ell-1} (x),\qquad S_1 (x) = -  \nabla \log p(x),
\#
\end{definition}

By this definition, note that the first-order score vector  and the  second-order score matrix  are  $S_1(x) = x$ and    $S_2 (x) =    I_d -  x x^\top $, respectively. Then by  \eqref{eq:def_score}, the  third-order score tensor~is  \#\label{eq:3rd_score}
S_3(x) = x \otimes x \otimes x - \sum _{j=1} ^d ( x \otimes e_j \otimes e_j + e_j \otimes x \otimes e_j + e_j \otimes e_j \otimes x),
\#
where $ \{ e_j \}_{j  \in [d]} \subseteq  \mathbb{R}^d$  are the standard basis for $\mathbb{R}^d$.
Based on  the score tensors, we now introduce the  moment tensors, whose decompositions reveal the  parameters of interest.

\begin{lemma}\label{lemma:moments}
For any $j \in [k]$, let $f_j \colon \RR^2 \rightarrow \RR$ be the link function in Definition \ref{def:sim0} or Definition \ref{def:sim}, and let $\epsilon_j$ be the exogenous random noise. We define 
$\varphi_j(u) = \EE_{\epsilon_j} [  f_j(u, \epsilon_j) ]$,  where the expectation is taken with respect to the randomness in $\epsilon_j$. We also define 
$\gamma_j^*  = \EE_{\xi\sim N(0,1)}[\varphi_j^{(\ell)}( \xi)]$, for any $j \in [k]$ and any positive integer $\ell$.
Then under assumption that $X$ is a standard Gaussian vector,  we have
\begin{align}
  \EE \bigl [ h_1(\mathcal{Z}) \cdot S_\ell(X) \bigr ]  = \frac{1}{k} \sum_{j=1}^k  \gamma_j^* \cdot {\beta_j^*}^{\otimes \ell},   \qquad \EE \bigl [h_2(\mathcal{Z}) \cdot S_\ell(X) \bigr ]  = \sum_{j=1}^k  \pi_j \cdot \gamma_j^* \cdot {\theta_j^*}^{\otimes \ell}. \label{cubemomenttensor}
\end{align}
\end{lemma}
\begin{proof}
This lemma follows by a straightforward application of the higher-order Stein's identity~\citep{janzamin2014score}.
\end{proof}

This  lemma suggests that  the parameters can be recovered by decomposing the sample versions of the  moment tensors in \eqref{cubemomenttensor}.
 The main reason for considering the higher-order moment tensors ($\ell \geq 3$), as opposed to second-order moment matrices ($\ell = 2$)  is that tensor decomposition is unique up to permutation and scaling~\citep{landsberg2011tensors}.
This allows us  to estimate the parameters  themselves as opposed to the subspace spanned by the parametric components. Furthermore,  tensor decomposition also allows one to work in the overcomplete setting ($k > d$). As mentioned in \S\ref{sec:intro}, this is particularly relevant for performing mixture of regression in big data settings, where the number of sub-populations in a dataset is typically large. While in theory, one can consider arbitrarily higher-order tensors, in the sample setting, they are notoriously  harder to estimate without relying on stringent model assumptions.  In this work, we consider specifically the case of $\ell=3$ and provide a detailed characterization of the rates of convergence for the proposed estimators. We discuss more about the theoretical results in case of $\ell >3$ as well as their pros and cons in Remark~\ref{rem:higherell}. 
Based on the above discussion, given $n$ samples $\{ (X^{(i)}, h(\mathcal{Z}^{(i)}))\}$, $i \in [n]$,  we can estimate the moment tensor in~\eqref{cubemomenttensor} by  
\#
\hat M_1 &= \frac{1}{n} \sum _{i=1}^n  h_1(\mathcal{Z}^{(i)}) \cdot S_3(X^{(i)}) = \frac{1}{n} \sum _{i=1}^n  \Big( \frac{1}{k} \sum_{j \in [k]} Z_j^{(i)} \Big) \cdot S_3(X^{(i)}),\label{eq:moment_tensor1}
\\
\hat M_2 &= \frac{1}{n} \sum _{i=1}^n  h_2(\mathcal{Z}^{(i)}) \cdot S_3(X^{(i)}) = \frac{1}{n} \sum _{i=1}^n  Y^{(i)} \cdot S_3(X^{(i)}). \label{eq:moment_tensor2}
\#
where $S_3 $ is the third order score function defined in \eqref{eq:3rd_score}. In what follows, we abuse our notation slightly and use $\hat{M}$ to denote both $\hat{M}_1$ and $\hat{M} _2$ whenever there is no confusion. Our estimators for $\{ \beta^*_j\}_{j\in [k]}$ and $\{  \theta_j^*\} _{j \in [k]} $ are then obtained by applying the tensor decomposition algorithms, described in \S\ref{sec:tpm} next, to  $\hat M_1$ and $\hat M_2$ respectively.


\subsection{Tensor Power Methods}\label{sec:tpm}
Now we construct estimators of the parametric components by applying the  tensor  power methods
 \citep{anandkumar2014tensor, anandkumar2014guaranteed, sun2017provable} to the moment tensor $\hat M$ under both the low and high-dimensional settings.

{\noindent \bf Low-Dimensional Setting.}
We apply the regular tensor power method to $\hat M$
in the low-dimensional setting where $n \geq d$, which is  an extension of the standard matrix power method  to tensors.
To simplify the notation, for any two vectors $u,v \in \RR^d$ and any third-order tensor $A \in \mathbb{R}^{d \otimes 3}$, we  denote the tensor-vector product between $A$ and $u, v$ by $A(I, u, v) \in \RR^d$, whose entries are specified by
\begin{align*}
[ A(I, u, v) ] _{j_1} = \sum_{j_2=1}^d \sum_{j_3=1}^d A_{j _1 j_2j_3} u_{j_2} v_{j_3}, \qquad j_1 \in [d].
\end{align*}
With this notation, the tensor power method is presented in Algorithm~\ref{alg:ocpower}. While the main intuition behind the tensor power method is similar to the matrix power method, there are delicate issues that arise solely for tensors. The main issue is that  perturbation results similar to the famous Davis-Kahn theorem for matrices  do not exists in the tensor setting. Recently, \cite{anandkumar2014tensor, anandkumar2014guaranteed}  establish the local and global convergence properties   of this algorithm. We leverage their results to establish statistical rates of convergence for our setting.

\begin{algorithm}[t]
\begin{algorithmic}
\STATE \textbf{Input:} Moment tensor $\hat{M} \in \mathbb{R}^{d \times d \times d}$, number of initializations $L$, number of iterations $N$.
\FOR {$\tau = 1, \ldots L$}
\STATE \textbf{Initialize} using unit vectors $\hat v^{(0)}_\tau$
\FOR {$t=0, \ldots, N-1$}
\STATE Perform power iteration
\begin{align}\label{tpiter}
\hat v^{(t)}_\tau = \frac{\hat M (I, \hat v^{(t-1)}_\tau, \hat v^{(t-1)}_\tau)}{\|\hat M (I, \hat v^{(t-1)}_\tau, \hat v^{(t-1)}_\tau) \|_2}
\end{align}
\ENDFOR
\ENDFOR 
\STATE Cluster the  $\{  \hat v^{(N)}_\tau\}_{\tau \in [L] }$ into $k$ clusters using the method in Algorithm~\ref{alg:clustep}
\STATE \textbf{Output:} The $k$ cluster centroids as $ \{ \hat \beta_j \}_{ j \in [k]} $ when $\hat M = \hat M_1$ and as $\{ \hat \theta_j \} _{ j \in [k]} $ when $\hat M = \hat M_2$
 \end{algorithmic}
\caption{Power Method for Overcomplete Tensor Decomposition}
\label{alg:ocpower}
\end{algorithm}


{\noindent \bf High-Dimensional Setting.}
Since the estimation error   using the regular tensor power method depends polynomially on the dimensionality, such a method is not applicable to the high-dimensional setting, where $n < d$ and the parametric components  are sparse.
To remedy this issue, we apply the truncated  tensor power method to $\hat M$ so as to leverage the sparsity, which is  analyzed in~\cite{sun2017provable}.
 Specifically, in each iteration, after a standard power iteration, we first  truncate the vector based on the top $ \bar s$ absolute  values of the current iterate, and then normalize the truncated iterate. This modification is given formally in  \eqref{ttpiter}. The overall algorithm  is presented for completeness in Algorithm~\ref{alg:octpower}.

\begin{algorithm}[t]
\begin{algorithmic}
\STATE \textbf{Input:} Moment tensor $\hat{M} \in \mathbb{R}^{d \times d \times d}$, number of initializations $L$, number of iterations $N$, rank $k$, and   sparsity  $\bar s$.
\FOR {$\tau = 1, \ldots L$}
\STATE \textbf{Initialize} using unit vectors $\hat v^{(0)}_\tau$
\FOR {$t=0, \ldots, N-1$}
{\STATE Perform a power iteration step
\$
\tilde v^{(t)}_\tau &= \frac{\hat M (I, \hat v^{(t-1)}_\tau, \hat v^{(t-1)}_\tau)}{\|\hat M (I, \hat v^{(t-1)}_\tau, \hat v^{(t-1)}_\tau) \|_2}
\$}
{\STATE Apply truncation to $ \tilde v^{(t)}_\tau$ and then do normalization
\begin{align}\label{ttpiter}
\hat v^{(t)}_\tau & = \frac{\vartheta_{\bar{s}}(\tilde v^{(t)}_\tau)} {\| \vartheta_{\bar s}(\tilde v^{(t)}_\tau)\|_2}
\end{align}}
\ENDFOR
\ENDFOR
\STATE Cluster the   $\{  \hat v^{(N)}_\tau\}_{ \tau \in [L] }$ into $k$ clusters using the method in Algorithm~\ref{alg:clustep}
\STATE \textbf{Output:} The $k$ cluster centroids as $\{ \hat \beta_j\}_{ j \in [k]}$ when $\hat M = \hat M_1$ and as $\{ \hat \theta_j\} _{ j \in [k]}$ when $\hat M = \hat M_2$
 \end{algorithmic}
\caption{Truncated Power Method for Overcomplete Sparse Tensor Decomposition}
\label{alg:octpower}
\end{algorithm}

Note that  the tensor power methods described in Algorithms~\ref{alg:ocpower} and \ref{alg:octpower}  involve post-processing in the form of clustering. Furthermore, the success of the algorithms also hinges  on the quality of  initialization, since  the objective function of tensor decomposition is nonconvex. We now discuss about these two crucial steps needed for the recovery of parametric components.

\noindent\textbf{Clustering.} Notice  that in both Algorithms~\ref{alg:ocpower} and \ref{alg:octpower}, the final step is a clustering procedure of the $L$ solution vectors $\{ \hat v_\tau\}_{ \tau \in [L] }$ (we drop the superscript based on $N$ here to avoid notational clutter). The main idea behind this clustering step is to estimate the $k$ components based on using the most correlated vectors from $\{ \hat v_\tau\}_{ \tau \in [L] }$  as initialization in the power method. The clustering procedure is outlined in Algorithm~\ref{alg:clustep}.

\begin{algorithm}[t]
\begin{algorithmic}
\STATE \textbf{Input:} $\hat{M} \in \mathbb{R}^{d \times d \times d}$, the set $\mathcal{S} = \{  \hat v_\tau, \tau \in [L] \}$.
\FOR {$j= 1, \ldots, k$}
\STATE Find $\hat v = \underset{v \in \mathcal{S}}{\argmax} | \hat M \odot (v \otimes  v \otimes v) |$ 
\STATE Perform $N$ iterates of  \eqref{tpiter} (resp.  \eqref{tpiter} and  \eqref{ttpiter}) for low-dimensional setting (resp. high-dimensional setting) with $\hat v$ as the initial point. Denote the final update as $\hat v_j$.
\STATE Remove all $\hat v_\tau \in \mathcal{S}$ such that $\| \hat v_\tau \pm \hat v \|_2 \leq 0.5$
 \ENDFOR
\STATE \textbf{Output:} The $k$ cluster centroids $\{\hat v_j\}_{j \in [k]}$.
 \end{algorithmic}
\caption{Clustering Step for (Truncated) Tensor Power Method}
\label{alg:clustep}
\end{algorithm}

\noindent \textbf{Initialization.} Obtaining a good initial point, satisfying the condition of required for the theoretical results in \S\ref{sec:theory} is a challenging task. In theory, provably good initialization could be obtaining based on unfolding and singular value decompositions when $k \leq c d$ where $c$ is an arbitrary constant. In practice, it has been observed in several works~\citep{anandkumar2014tensor, sun2017provable} that random initialization works well even in the overcomplete setting. However,  obtaining a theoretical statement quantifying this observation has remained elusive so far. We note that, even with random initialization, the number of initialization $L$ needs to be set for both Algorithm~\ref{alg:ocpower} and~\ref{alg:octpower}.







\section{Main Results}\label{sec:theory}

In this section, we state our main result for estimating the parameters $\{ \beta_j ^*\}_{ j\in [k]}$   and  $\{ \theta_j^* \}_{j \in [k]} $ of the two models in Definitions~\ref{def:sim0} and~\ref{def:sim} respectively. Our proofs consists of two parts. We first leverage a deterministic results (i.e., results deterministic up to randomness in the algorithm's initialization) concerning the convergence of tensor power method (resp. truncated tensor power method) from~\cite{anandkumar2014guaranteed} (resp. from~\cite{sun2017provable}). Specifically, for the low-dimensional case,
such a deterministic convergence result relates the statistical performances of  the estimators constructed by Algorithm \ref{alg:ocpower} to $\| \hat M - \EE [ Y \cdot S_3(X)] \|_{\oper}$, where $\hat M = \hat M_1$ and $Y = h_1 (\cZ)$ for the discordant SIMs, and $\hat M = \hat M_2$ and $Y = h_2 (\cZ)$ for mixture of SIMs. Similarly, for the high-dimensional setting, the deterministic result in \cite{sun2017provable} bounds the statistical error of the estimators in Algorithm \ref{alg:octpower} by $\| \hat M - \EE [ Y \cdot S_3(X)] \|_{\oper,r}$.  Our major contribution in this work is obtaining high-probability concentration bounds for both $\| \hat M - \EE [ Y \cdot S_3(X)] \|_{\oper}$ and $\| \hat M - \EE [ Y \cdot S_3(X)] \|_{\oper,r}$, which might be of interest to other models estimated using method-of moments. Compared to the matrix concentration results, obtaining concentration results for higher-order tensors are significantly challenging. The main difficulty is obtaining sharp concentration bounds for certain polynomial functions of random variables, that enables one to leverage the $\epsilon$-netting combined with union bound argument. We refer the reader to the proofs in \S\ref{sec:aux} for the details. Before we state and discuss our main results, we first outline the assumptions we make in this work, which can be classified into two types that   correspond to the probabilistic and deterministic parts of our proof. While the probabilistic assumptions are the same for both the low- and high-dimensional cases, the deterministic assumptions on the parameters vary for the low and high-dimensional settings.

\begin{assumption}[Probabilistic Assumptions] \label{assume:moments}
For the discordant SIMs in Definition \ref{def:sim0} and mixture SIMs in Definition \ref{def:sim},  we assume that the following   conditions are satisfied.
\begin{itemize}
\item[1.1] \textrm{Noise Assumption.} The noise $\epsilon_1, \ldots, \epsilon_k$ are such that the response  $Y$ is a sub-exponential random variable with $\| Y \|_{\psi_1} \leq \Psi$, where $Y = h_1(\cZ)$ for  discordant SIMs and $Y = h_2( \cZ) $ for mixture of SIMs.
\item[1.2] \textrm{Covariate Assumption.} The covariate $X\sim N(0,I_d)$ is a Gaussian random vector.
\item [1.3] \textrm{Regularity Assumption.} Recall that we define $\{\gamma _j^* \}_{j \in [k]}$ in Lemma \ref{lemma:moments}. For discordant SIMs, we assume that there exists $\gamma _{\max} , \gamma _{\min} >0$ such that $ \{ \gamma^* _j  / k \} _{j \in [k]} \subseteq [ \gamma_{\min}, \gamma_{\max} ] $. In addition, for mixture SIMs, we assume that $ \{ \gamma^* _j \cdot \pi_j \} _{j \in [k]} \subseteq [ \gamma_{\min}, \gamma_{\max} ] $.
\end{itemize}
\end{assumption}
The assumption that $Y$ is sub-exponential is a substantially weaker condition allowing for potentially heavy-tailed noise. Relaxing this assumption would incur a significant loss in the rates of convergence of $ \| \hat M - \EE [ Y \cdot S_3(X)] \|_{\oper}$ and $ \| \hat M - \EE [ Y \cdot S_3(X)] \|_{\oper,r}$, which consequently leads to slower rates of convergence for estimating the parameters $\{ \beta_j ^*\}_{ j\in [k]}$   and  $\{ \theta_j^* \}_{j \in [k]} $. The assumption that  $X$ has  i.i.d Gaussian entries could be relaxed to the case of $X \in N(0,\Sigma)$ with a well-conditioned $\Sigma$. Such an assumption is standard in several works on estimating functionals of covariance matrices; see for example~\citep{cai2016estimating}. We do not explicitly concentrate on the relaxed assumption so as to highlight the main message of the paper in a simpler setting. Relaxing the assumption on $X$ further to non-Gaussian distributions  is rather delicate, which   is further  discussed  in \S\ref{sec:ending}. Roughly in this setting, either more structure should be enforced on the parameters,  or more information about the density of $X$ must be known. We now state our results for the low and high-dimensional setting in Subsections~\ref{ref:ldmain} and~\ref{ref:hdmain} respectively.


\subsection{Low-dimensional Results}\label{ref:ldmain}

We first characterize the statistical rate of convergence for $\hat\beta_j$ and $\hat \theta_j$ in the low-dimensional setting. As mentioned previously, in order to obtain the estimation rates stated in Theorem~\ref{thm:ldrates}, we require concentration bounds on $\| \hat M_1 - \EE [ h_1(\mathcal{Z}) \cdot S_3(X)] \|_{\oper}$ and $\| \hat M_2 - \EE [ h_2(\mathcal{Z}) \cdot S_3(X)] \|_{\oper}$. We state the result below.
\begin{theorem} \label{thm:tensor_concentration}
Note that we define tensors   $\hat M_1$ and $\hat M_2$   in \eqref{eq:moment_tensor1} and \eqref{eq:moment_tensor2}, respectively.   Under Assumption \ref{assume:moments}, when  $n \cdot  d$ is   sufficiently large,
with probability at least $1 -\exp(-2d)$,  we have
 \#\label{eq:tensor_concentration}
\bigl \| \hat M_1 - \EE [ h_1(\mathcal{Z}) \cdot S_3(X)]  \bigr \|_{\oper} \leq  K \max  \Bigg( \sqrt{\frac{d}{n}},  \frac{d^{5/2}}{n}  \Bigg)
 \#
 where $K$ is an absolute constants. The same bound holds for  $\bigl \| \hat M_2 - \EE [ h_2(\mathcal{Z}) \cdot S_3(X)] \bigr \|_{\oper}$.  
\end{theorem}
\noindent The above theorem, establishes concentration of  $\hat M_1 - \EE [ h_1(\mathcal{Z}) \cdot S_3(X)]$ and $\hat M_2 - \EE [ h_2(\mathcal{Z}) \cdot S_3(X)]$ in tensor operator norm.  The proof of the above theorem is involved and is deferred to the appendix. We now proceed to state the results for estimation error. In order to do so, apart from Assumption \ref{assume:moments}, we make the following deterministic assumptions on the true parameters $\{ \beta_j^*\}_{j \in [k]}$ and $\{ \theta_j^* \}_{j \in [k]}$ corresponding to the two models in \S\ref{sec:models}. In what follows, absolute constants are denoted by $C$ or $K$ indexed with a subscript. The values of the constants may change from line to line.

\begin{assumption}[Low-dimensional Deterministic Assumptions] \label{assume:parametersld}
Let  $U= [ u_1, \cdots, u_k] \in \RR^{d \times k}$ be a matrix with vectors $\{ u_j \}_{ j \in [k]} \subseteq \RR^d$  as its columns. We introduce the following two conditions on $\{ u_j \}_{ j \in [k]}$, which are assumed to be satisfied by both  $\{ \beta_j^* \}_{j \in [k]}$ and $\{ \theta_j^* \}_{j\in [k]}$.
\begin{enumerate}
\item[2.1] \text{Incoherence condition.}  There exist absolute constants $C_0$ and $C_1$ such that
 $$ \psi = \max_{i \neq j} | u_i^\top u_j|  \leq C_0/ \sqrt{d} \qquad\text{and}\qquad \| U\|_{\oper} \leq 1+ C_1\sqrt{k/ d}.$$

 \item [2.2] \textrm{Overcompleteness.} The number of SIM, $k = o(d^{3/2})$.
\end{enumerate}
\end{assumption}
The incoherence assumption is a standard condition in the literature on high-dimensional statistics literature~\citep{donoho2001uncertainty,donoho2006stable} and is particularly common for theoretical analysis of tensor decomposition~\citep{anandkumar2014tensor, anandkumar2014guaranteed, sun2017provable}. It is a relaxation of more restrictive orthogonality condition and allows for a much broader class of parameter vectors in the DAIM models we consider. Relaxing such an assumption is significantly harder and may lead to inefficient estimation rates.  Furthermore, to characterize the performance of our estimator, note that $\beta^*_j$ and $\theta_j^* $ are unidentifiable in the DAIM  since $f_j$  in \eqref{eq:discoaim} is unknown. Thus,  for
 two vectors $u_1$ and $u_2 $, we use
 \#\label{eq:sign_flip_distance}
 d(u_1, u_2 )=\min \bigl \{ \| u_1 - u_2  \|_2,  \| u_1 + u_2 \|_2 \bigr \}
 \#
 to
 measures the distance between $u_1$ and $u_2$ up to sign-flips, which is used   to evaluate the performance of the estimator. Additionally, we define the following two parameters that characterize the rates of convergence and  the requirement for initialization respectively:
\begin{align*}
\varrho ( \hat M, \gamma) & =  \frac{2\sqrt{5}}{\gamma_{\min}} ~\bigl \|  \hat M - \EE [ Y \cdot S_3(X)] \bigr\|_{\oper} + \frac{2\sqrt{5} C_1 \gamma_{\max}}{\gamma_{\min}} \sqrt{k} \psi^2,\\
\varrho_0 (\gamma) & =  \min \bigg[  \sqrt{\frac{\gamma_{\min}}{6 \gamma_{\max}}},    \frac{\gamma_{\min}  }{4 \gamma_{\max}} - \frac{C_1 \sqrt{k}}{d}, \frac{\gamma_{\min}}{4\sqrt{5} C_2 \gamma_{\max}} - \frac{2C_0}{C_2 \sqrt{d}} \bigg( 1 + C_1 \sqrt{\frac{k}{d}} \bigg)^2  \bigg].
\end{align*}
Note that in the above definition, $Y = h_1(\mathcal{Z})$ when $\hat M = \hat M_1$ and  $Y = h_2(\mathcal{Z})$ when $\hat M = \hat M_2$. With the above notation, we now state our theorem for estimation rates.

\begin{theorem}[Rates in Low-dimensions] \label{thm:ldrates}
For  discordant SIMs  in Definition~\ref{def:sim0}  and mixture SIMs in Definition \ref{def:sim}, let $\{ \hat \beta_j \}_{j \in [k]}$ and $\{ \hat \theta_j \}_{j\in [k]}$ be the estimators returned by  Algorithm~\ref{alg:ocpower} with $\hat{M}_1$ and $\hat{M}_2$ as inputs, respectively.
Under Assumptions \ref{assume:moments} and  \ref{assume:parametersld}, we assume  the number of iterations satisfy
 $N_1 \geq C_3 \log\{\gamma_{\min}/[\gamma_{\max} \cdot \varrho(\hat{M}_1, \gamma)]\}$ and $N_2 \geq C_4 \log\{\gamma_{\min}/[\gamma_{\max} \cdot \varrho(\hat{M}_2, \gamma)]\}$ for  $\hat M_1$ and $\hat M_2$ are used respectively. Then with probability tending to 1, for any $j \in [k]$, we have
\begin{align}\label{mainldrate}
 d(\hat \beta_j, \beta^*_j) \leq  \frac{2\sqrt{5}}{\gamma_{\min}}\Biggl[  K  \max  \bigg( \sqrt{\frac{d}{n}}, \frac{d^{5/2}}{ n} \bigg)\Biggr] + \frac{2\sqrt{5} C_1 \gamma_{\max}\sqrt{k} \psi^2}{\gamma_{\min}}
\end{align}
as long as the initialization $\hat \beta^{(0)}_\tau $ satisfies respectively $d(\hat \beta^{(0}_\tau, \theta^*_j) \leq \varrho_0(\gamma)$. The same results also holds for $d(\hat \theta_j, \theta^*_j)$ as long as the corresponding initialization satisfies $d(\hat \beta^{(0}_\tau, \beta^*_j) \leq \varrho_0(\gamma)$.
\end{theorem}
\begin{proof}
The proof of the theorem follows immediately by combining the statements of Theorem 1 in ~\cite{anandkumar2014guaranteed} and part (a) of Theorem~\ref{thm:tensor_concentration} on tensor norm concentration, proved in Appendix~\ref{sec:proof_thm1}.
\end{proof}

\begin{remark}
The above theorem has two terms that characterize the rates of convergence of $\hat \beta_j$ to $\beta_j^*$ and $\hat \theta_j$ to $\theta_j^*$. The first term in \eqref{mainldrate} is essentially the error  of estimating the third-order moment tensors  in~\eqref{cubemomenttensor} using the empirical tensors in \eqref{eq:moment_tensor1} and \eqref{eq:moment_tensor2}. Moreover, such a estimation error has different behaviors in the low-sample regime ($n \leq d^4$) and in the high-sample regime ($n \geq d^4$). Specifically, in the low-sample regime, the rate is dominated by the slower $d^{5/2}/n$ term; whereas in the high-sample regime it is dominated by the faster rate  $\sqrt{d/n}$. Hence, with big data,  the statistical rate of convergence can be much faster. In addition, the second term  in \eqref{mainldrate} could be interpreted as the approximation error term, which arises from the analysis of the tensor power method for overcomplete tensor decomposition~\citep{anandkumar2014guaranteed}. The incoherence condition in Assumption~\ref{assume:parametersld} leads to the constraint that $k = o(d^{3/2})$ for consistency. This essentially controls the level of overcompleteness in the model for consistent estimation of the parameters.
\end{remark}

Finally, we note that in the context of mixture of generalized linear models,~\cite{sedghi2016provable} presented a theorem on statistical rates of convergence for a related estimator. Unfortunately, the presented rates are highly sub-optimal in comparison and no proofs are provided. 

\subsection{High Dimensional Results}\label{ref:hdmain}
Similar to the low-dimensional setting, we now present sparse tensor operator norm bounds that are required to obtain the estimation error rates in the high-dimensional setting.


\begin{theorem} \label{thm:tensor_concentration1}
Under  Assumption~\ref{assume:moments}, when $n \cdot d$ is sufficiently large, with probability at least $1 - 8 \exp[ - 3r \log (d) ] $, we have
 \#
 \bigl \| \hat M_1 - \EE [ h_1(\mathcal{Z}) \cdot S_3(X)]  \bigr \|_{\oper,r} \leq K \max \Bigg \{     \sqrt{\frac{ r \log (d )}{ n}},   \frac{[ r \log (d) ]^{5/2}}{n}  \Bigg \}
\label{eq:stopt2}
 \#
 where $r$ is a positive integer (typically much less than $d$) and $K$ is an absolute constant. The same bound also holds for $\bigl \| \hat M_2 - \EE [ h_2(\mathcal{Z}) \cdot S_3(X)] \bigr \|_{\oper,r}$.
\end{theorem}

The proof of the above theorem is deferred to the appendix. We now state the following deterministic assumption on the parametric components required in the high-dimensional setting.
\begin{assumption}[High-dimensional Deterministic Assumptions] \label{assume:parametershd}
Let $\{ u_j \}_{ j \in [k]}$ be a set of vectors in $\RR^d$  and let $U= [ u_1, \cdots, u_k]$ be the matrix with the vectors as its columns. The following  three conditions on $\{ u_j \}_{ j \in [k]}$  are assumed  to hold for  both  $\{ \beta_j^* \}_{j \in [k]}$ and $\{ \theta_j^* \}_{j\in [k]}$.
\begin{enumerate}
\item \textrm{Sparsity.} The vectors $u_j$ has at most $s$ non-zero entries, i.e.,  $\|u_j^* \|_0 \leq s$ for any $j \in [d]$.
\item \textrm{Incoherence condition.}  There exist absolute constants $C_5$ and $C_6$ such that
 $$ \psi = \max_{i \neq j} | u_i^\top u_j|  \leq C_5 / \sqrt{s}  \qquad\text{and}\qquad \| U\|_{\oper} \leq 1+ C_6\sqrt{k /s}.$$
\item \textrm{Overcompleteness.} The number of mixture components $k = o (d^{3/2})$.
\end{enumerate}
\end{assumption}
\noindent While the form of the incoherence and overcompleteness conditions are same as in the low-dimensional setting, we additionally assume that the parametric components are $s$-sparse. Estimation in this setting corresponds to decomposing the moment tensors in \eqref{cubemomenttensor}  into sparse factors. For this problem, \cite{sun2017provable} proposed a sparse tensor power method for such a decomposition, outlined in Algorithm~\ref{alg:clustep}. We now state the main result of this section based on the notations below. Analogous to the low-dimensional setting, we define
\begin{align*}
\phi(\hat M, \gamma) & =  \frac{2\sqrt{5}}{\gamma_{\min}} \|  \hat M - \EE [ Y \cdot S_3(X)]\|_{\oper, (s+\bar s)} + \frac{2\sqrt{5} C_6 \gamma_{\max}}{\gamma_{\min}} \sqrt{k} \psi^2,\\
\phi_0(\gamma)& =  \min \Bigg[~\frac{\gamma_{\min}}{6 \gamma_{\max}} - \frac{C_6 \sqrt{k}}{s}, \frac{\gamma_{\min}}{4\sqrt{5} C_7\gamma_{\max}} - \frac{2C_5}{C_7 \sqrt{s}} \Bigg( 1 + C_6 \sqrt{\frac{k}{s}} \Bigg)^2 ~\Bigg],
\end{align*}
where $Y = h_1(\mathcal{Z})$ when $\hat M = \hat M_1$ and  $Y = h_2(\mathcal{Z})$ when $\hat M = \hat M_2$.
Now we are ready to  state the estimation rates for  the high-dimensional setting.

\begin{theorem}[Rates in High-dimensions]\label{thm:hdrates}
For  the two  models in Definitions~\ref{def:sim0} and~\ref{def:sim},
let $\{ \hat \beta_j \}_{j \in [k]}$ and $\{ \hat \theta_j \}_{j\in [k]}$ be the estimators returned by  Algorithm~\ref{alg:octpower} with $\hat{M}_1$ and $\hat{M}_2$ as inputs, respectively.
Under Assumptions \ref{assume:moments} and  \ref{assume:parametershd}, we assume  the number of iterations satisfy
$N \geq C_8 \log(\phi_0(\gamma)/\phi(\hat M_1,\gamma))$ and $N \geq C_9 \log(\phi_0(\gamma)/\phi(\hat M_2,\gamma))$ for  $\hat M_1$ and $\hat M_2$ are used respectively. Then with probability tending to 1, for any $j \in [k]$,  
$d(\hat \beta_j, \beta^*_j)$ is upper bounded by
\begin{align*}
 \frac{2\sqrt{5}}{\gamma_{\min}} \Biggl( K_{3} \cdot \max \biggl \{ \biggl [ \frac{(s+\bar s) \log (d )}{n} \biggr ]^{1/2}, \frac{ [  (s+\bar s) \log (d) ]^{5/2}} {n} \biggr\}\Biggr) +\frac{2\sqrt{5} C_1 \gamma_{\max} \sqrt{k} \psi^2}{\gamma_{\min}}
\end{align*}
as long as the initialization $\hat \beta^{(0)}_\tau $ satisfies   $d(\hat \beta^{(0}_\tau, \theta^*_j) \leq \phi_0(\gamma)$. The same bound also holds for  $d(\hat \theta_j, \theta^*_j) $ as long as the corresponding initialization satisfies $d(\hat \beta^{(0}_\tau, \beta^*_j) \leq \phi_0(\gamma)$.\end{theorem}

\begin{proof}
The proof of the theorem follows immediately by combining the statements of Theorem 3.6 in ~\cite{sun2017provable} and our result on sparse tensor concentration in  Theorem~\ref{thm:tensor_concentration1}, proved in Appendix~\ref{sec:proof_thm2}.
\end{proof}
\begin{remark}
Similar to the low-dimensional case, in Theorem \ref{thm:hdrates}, the first two terms in the statistical rate characterizes the  error of  estimating the moment tensors in \eqref{cubemomenttensor} using the empirical tensors in \eqref{eq:moment_tensor1} and \eqref{eq:moment_tensor2}. Thanks to the  sparsity assumption, such an  estimation error depends only poly-logarithmically on  dimensionality $d$.  Indeed, if we pick $\bar s = \Theta(s)$, then the estimation error is predominantly controlled by a polynomial in $s/n$. Similar to the previous case, the estimation error has different behaviors in the low-sample regime ($n \leq s^4$) and in the high-sample regime ($n \geq s^4$). Up to poly-logarithmic terms in $d$, in the low-sample regime, the rate of convergence is dominated by the slower $s^{5/2}/n$ term and in the high sample regime it is dominated by  $\sqrt{s/n}$. Recall that in the low-dimensional case, the sample complexity and the rate of convergence are crucially dependent on $d$ and hence is not feasible for the high-dimensional situations. Indeed in the high-dimensional setting our estimator leverages the structural sparsity assumption, as is commonly done in the literature on high-dimensional statistics, to get a milder dependence on the dimensionality.
\end{remark} 

\begin{remark}\label{rem:higherell}
{As discussed in \S\ref{sec:models}, our theoretical results in both \S\ref{ref:ldmain} and~\S\ref{ref:hdmain} are detailed for the case of third-order moment tensor decompositions (i.e., $\ell =3$). It is indeed possible to easily extend our results for general $\ell$-th order decomposition. In this case, the overcompleteness assumption could be relaxed to $k = o(d^{\ell/2})$ allowing for a wider class of parametric components. Indeed this comes at the cost of requiring more samples to estimate a higher-order moment tensor accurately. Specifically, for the estimation error, one would obtain
\begin{align*}
\max  \bigg( \sqrt{\frac{d}{n}}, \frac{d^{(\ell+2)/2}}{n} \bigg) \quad \text{and} \quad \max  \bigg( \sqrt{\frac{(s+\bar{s}) \log(d)}{n}}, \frac{[(s+\bar{s})\log(d)]^{(\ell+2)/2}}{n} \bigg)
\end{align*}
respectively in Theorem~\ref{thm:ldrates} and~\ref{thm:hdrates}. In order to obtain the above rates, the main modification   required is to re-derive the concentration result in Lemma~\ref{lemma:sum_sub_exponential} for this setting. Seen from the proof of this lemma in \S\ref{sec:concentration_results}, for the $\ell$-th order moments, we could similarly construct i.i.d. sub-Gaussian random variables and apply Theorem \ref{thm:subgauss_poly} to obtain the desired concentration results.}
\end{remark}




 \section{Numerical Experiments}

In this section, we evaluate the finite-sample performances of the proposed estimators  via  numerical simulations. Without loss of generality, we only present the results for the discordant SIMs in Definition \ref{def:sim0}; the performance of mixture of SIMs are similar.  We consider both the low- and high-dimensional settings.
Throughout the experiments, for the $k$ latent single index models   $\mathcal{Z} = \{Z_j = f_j (\la \beta_j^*, X \ra, \epsilon_j ) \}_{j \in [k]}$, we set the link functions  to  be the same for simplicity.
Specifically, for any $j \in [j]$, we let
$f_j (u,v) = h(u) + v$, where $h \colon \RR\rightarrow \RR$
 is    one of the following three univariate functions:
\#\label{eq:exper_links}
h_1 (u ) &=  u^3 + 10 \cdot \exp(-u^2) \\ \nonumber
h_2 (u) &= u^3 + 5 \cdot \sin (2\cdot u^2) \\ \nonumber
 h_3 (u) &=u^3   + 10\cdot \tanh(u^2) . 
\#
In addition, we  let $\{ \epsilon_j \}_{j \in [k]}$   be i.i.d. Gaussian  random variables with mean zero and variance $1/k$, and
  set $X \sim N(0, I_d)$.
  Finally,  for  the low-dimensional case,  the signal parameters  $\{ \beta_j^* \}_{j \in [k]}$ are generated as follows. We   let $\{ v_j \}_{j \in [k]}$ be $k$ orthornormal vectors in $\RR^d$ and let $\{ e_j \}_{j \in [k]}$ be $k$ i.i.d. $N(0, I_d)$ random vectors. Then we define each $\beta_j^* $ by $v_j + \kappa \cdot e_j / \|  v_j + \kappa \cdot e_j  \|_2$, where   $\kappa >0$  is a small constant  chosen such that 
  the signal parameters satisfy the incoherence condition given in Assumption \ref{assume:parametersld}. 
   In addition, for the high-dimensional case, note that the signal parameters has $s$ nonzero entries. 
   We first generate $s$ incoherent vectors, denoted by $\{ u_j\}_{j \in [s]} \subseteq \RR^s$, in the same fashion as in the low-dimensional setting. Let $m = \lceil k / s \rceil$. We generate $m$ disjoint subsets $\cR_1, \ldots, \cR_m$ of 
$[d]$ with cardinality $r$ randomly. For any $j \in [k]$, suppose $j$ can be written as $j  = p m + q$ where $p, q \geq 0$ and $q < m$.  Then we let the support of $\beta_j^* $ be $\cR_{p+1}$, and let it be  $  u_{q+1}$     when restricting on the support. 
One could easily  verify that $\{ \beta_j^*\}_{j \in [k]}$ defined in this way satisfy Assumption~\ref{assume:parametershd}.

Furthermore, in the low-dimensional setting, for all the experiments, we set $d = 20$, $k \in \{3, \ldots, 7\}$, and let $n$ vary. Let $\{\hat \beta_j \}_{j \in [k]}$ be the final estimators returned by Algorithm \ref{alg:ocpower} with the input moment tensor $\hat M$ given  in \eqref{eq:moment_tensor1}. We set the number of initializations and iterations to be $L = 200$ and $N = 300$, respectively.  We access the estimation performance by computing   $\max_{j\in [k]} \{ d(\hat \beta_j , \beta_j^*) \}$, where the distance function  $d $ is defined in \eqref{eq:sign_flip_distance}. In   Figure \ref{fig:simulation}, we plot the estimation error against the inverse signal strength
$
\max\{ \sqrt{d/n}, d^{5/2} / n \}
$
for all the three link functions in \eqref{eq:exper_links},  based on $100$ independent trials for each $(n, d, k)$.
As shown in Theorem \ref{thm:ldrates},  the estimation error is bounded by a linear function of $\max\{ \sqrt{d/n}, d^{5/2} / n \} $.  Moreover,  the slope of this linear function does not depends on $n$, $d$, or $k$, and the intercept is bounded by $C\cdot \sqrt{k}\cdot \psi^2$ for some constant $C>0$, where $\psi$ is the incoherence parameter defined in \eqref{ref:hdmain}. As shown in Figure \ref{fig:simulation}, all the curves of estimation errors are below a straight line with positive slope and intercept, which corroborates   the statistical rates in the low-dimensional  settings established in Theorem \ref{thm:hdrates}.

Similarly, in the high-dimensional regime, we set $d = 100$, $s \in \{ 3, 4,5\}$, $k \in \{ 3, \ldots, 7\}$, and let $n$ vary. In this case, we also report the estimation error $\max_{j\in [k]} \{ d(\hat \beta_j , \beta_j^*) \}$, where $\{\hat \beta_j\}_{j \in [k]}$ are the output of Algorithm \ref{alg:octpower} with the input moment tensor $\hat M$ given  in \eqref{eq:moment_tensor1}. For the hyperparameters of Algorithm \ref{alg:octpower}, we set $L = 100$, $N = 200$,  and   $ \bar s = 3  s$ in all experiments, where $\bar s$ is the parameter of the truncation step. Moreover,  as  suggested by Theorem \ref{thm:hdrates}, we define the inverse signal strength  by $\max \{ \sqrt{s\log d/n} , ( s \log d)^{5/2} / n \}$, which reflects  the theoretical estimation accuracy. We plot the estimation error against the inverse signal strength in Figure  \ref{fig:simulation2}. As shown in the figures, the estimation errors are all bounded by a  linear function of the   inverse signal strength and are not sensitive to the choice of $k$,    which  is suggested by Theorem \ref{thm:hdrates}.  Specifically, the three plots in the first row correspond to the results for $s=3 $, and the case of $s = 4$ and $s = 5$ are reported in the second and last row respectively. The three columns correspond to link functions $h_1$, $h_2$, and $h_3$, respectively.

  \begin{figure}[h!]
	\centering
	\begin{tabular}{ccc}
		\hskip-30pt \includegraphics[width=0.35\textwidth]{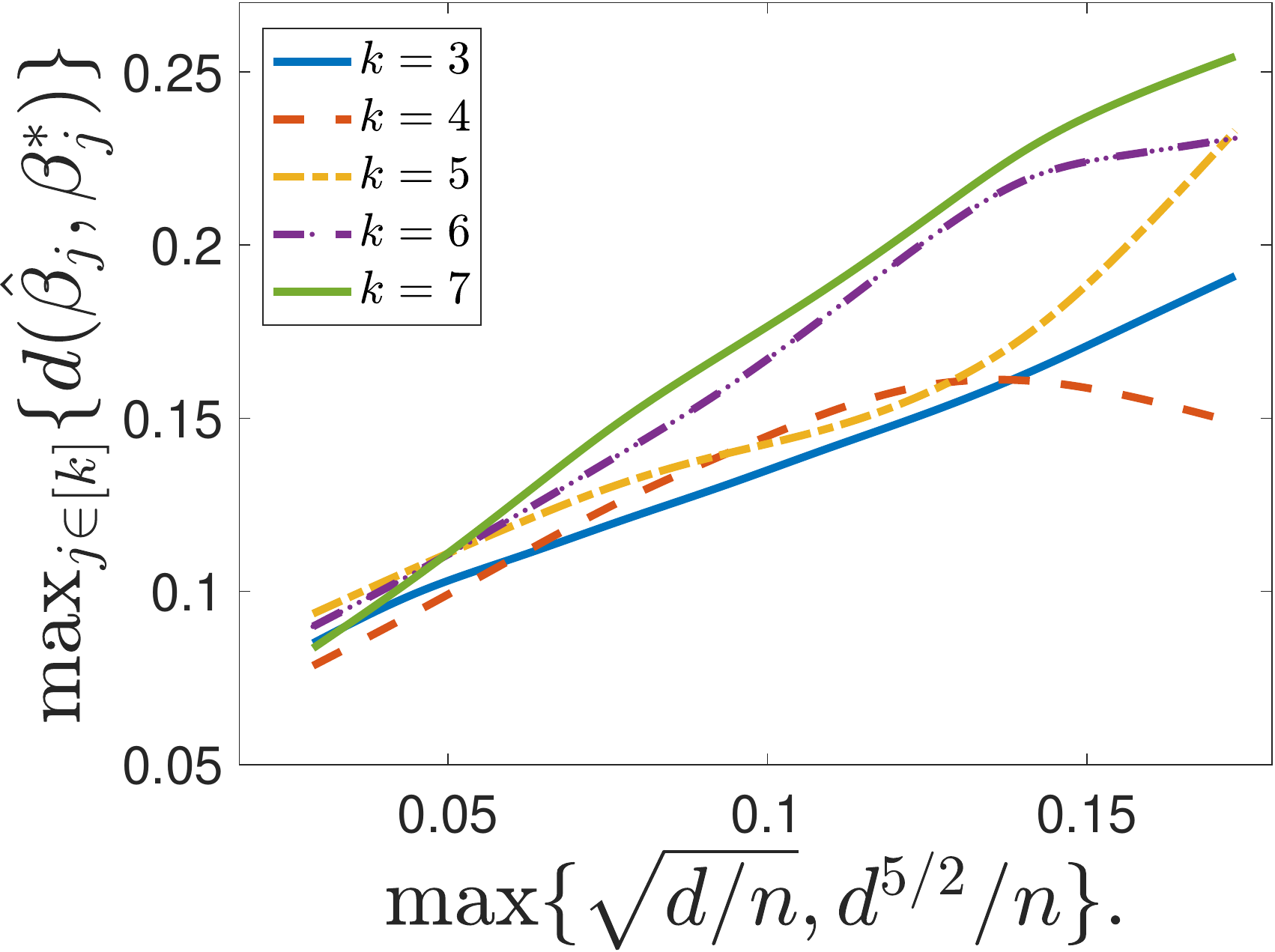}
		&
		\hskip-30pt \includegraphics[width=0.35\textwidth]{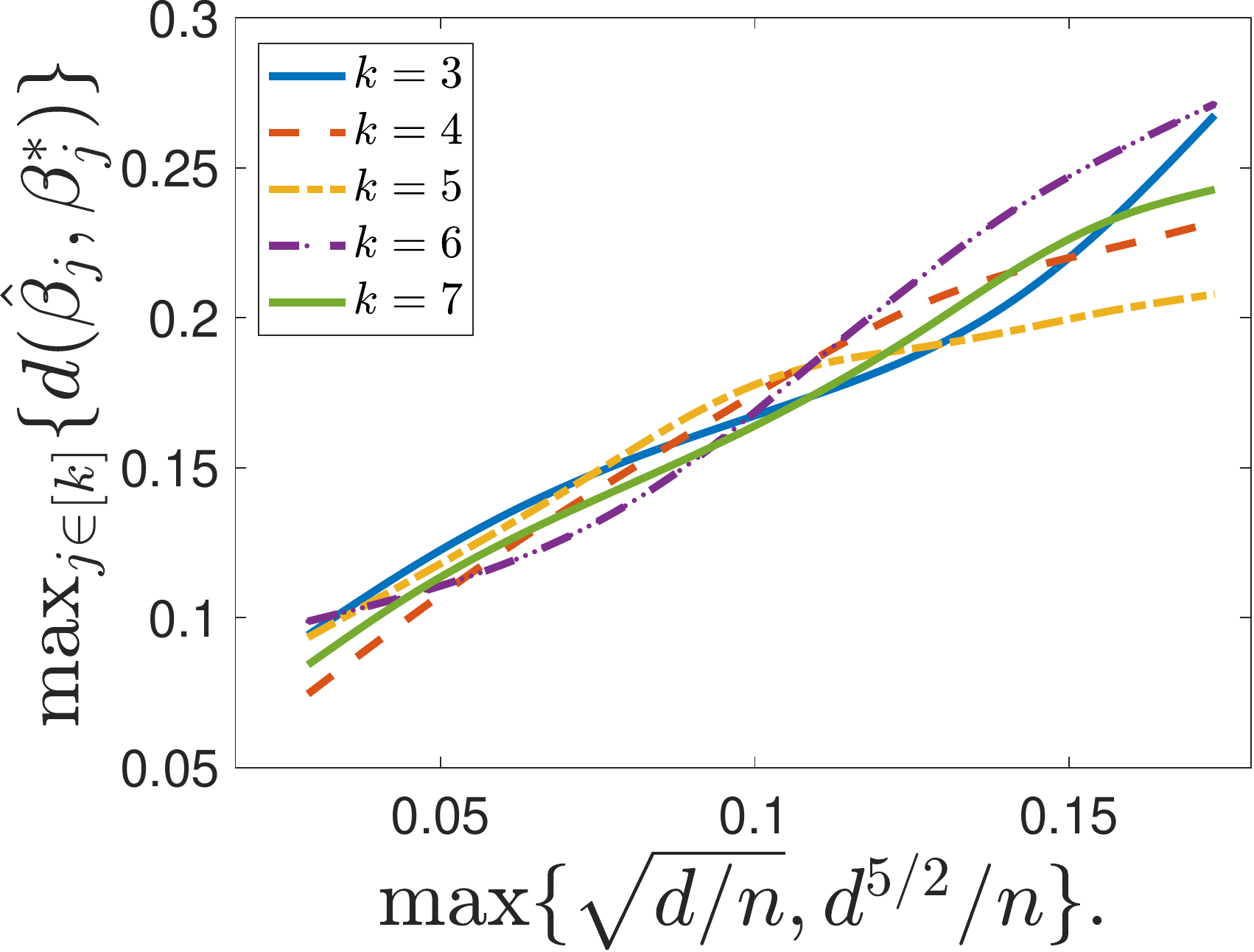}
		 &\hskip-20pt  \includegraphics[width=0.35\textwidth]{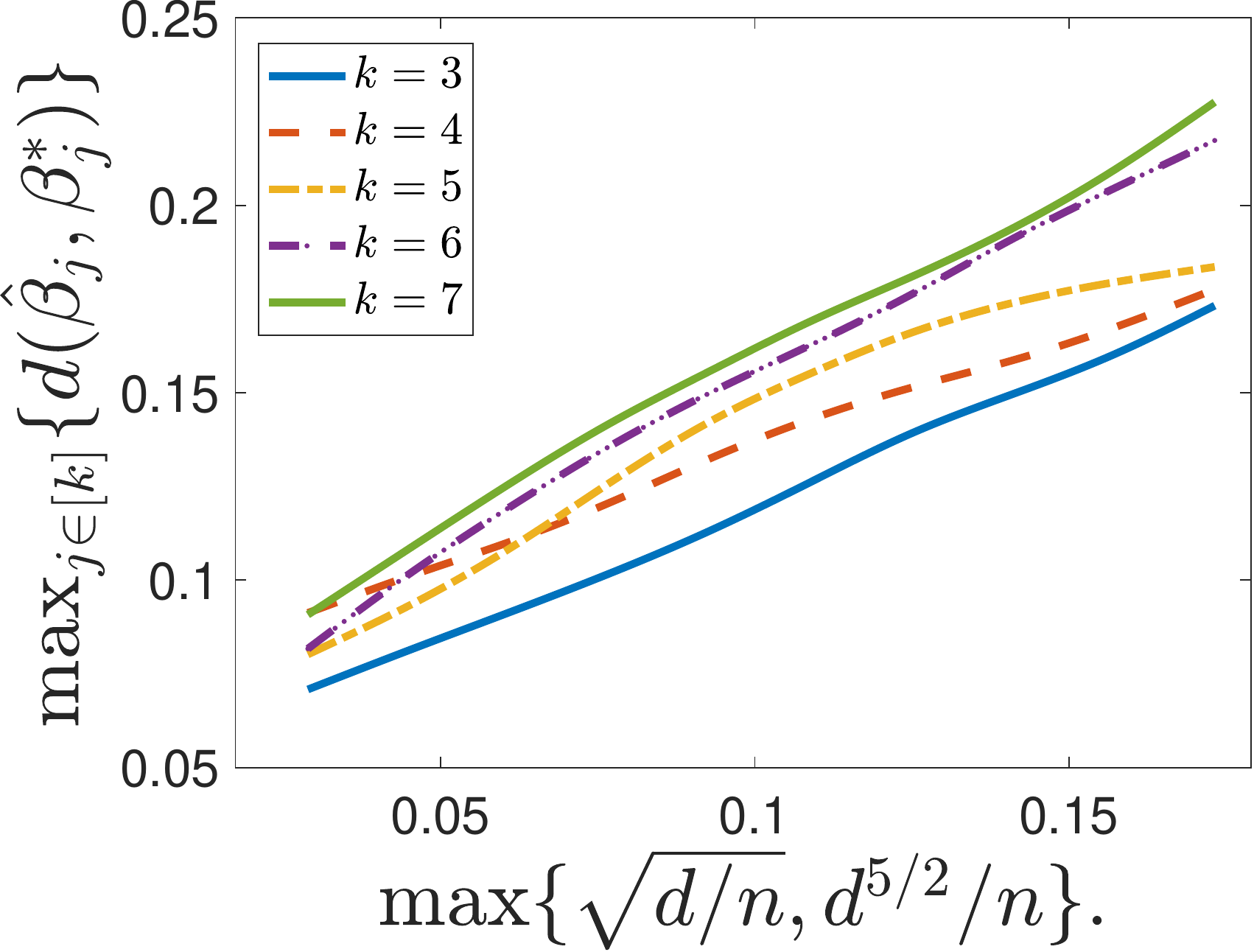} \\
		(a) $h_1(u)  = u^3 + 10 \cdot  \exp( - u^2)  $ & (b) $ h_2(u)  = u^3 + 5 \cdot \sin (2\cdot u^2)$, & (c) $ h_3(u)  =u^3 + 10 \cdot \tanh(u^2)$  	
	\end{tabular}\\
	\caption{\small \em Plots of the  estimation error    $\max_{j\in [k]} \{ d(\hat \beta_j , \beta_j^*) \}$ against the inverse signal strength $\max\{ \sqrt{d/n}, d^{3/2} / n \}$,  in which $\{\hat \beta_j\}_{j \in [k]} $ are  the output of Algorithm \ref{alg:ocpower} with the input moment tensor $\hat M$ defined in \eqref{eq:moment_tensor1}.   The  link function is    one of $h_1$, $h_2$, and $h_3$ in \eqref{eq:exper_links}. The three columns correspond to $h_1$, $h_2$, and $h_3$, respectively. Moreover, we set  $d = 20 $, $k \in \{3, \ldots, 7\}$,  and let $n$ vary. We generate each   figure   based on $100$ independent trials for each $(n, d, k)$.  }
	\label{fig:simulation}
\end{figure}

  \begin{figure}[h!]
	\centering
	\begin{tabular}{ccc}
		\hskip-10pt \includegraphics[width=0.35\textwidth]{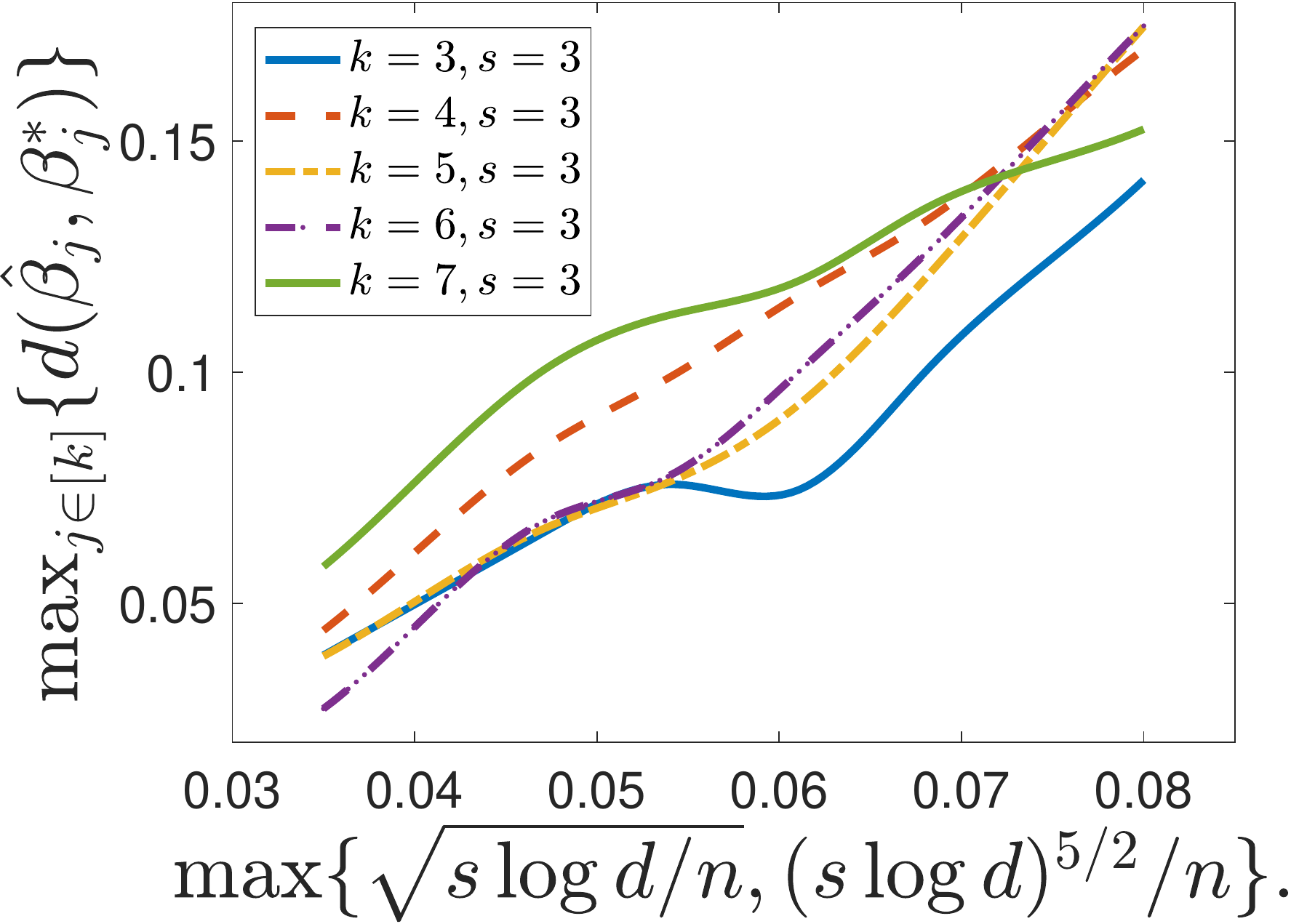}
		&
		\hskip-10pt \includegraphics[width=0.35\textwidth]{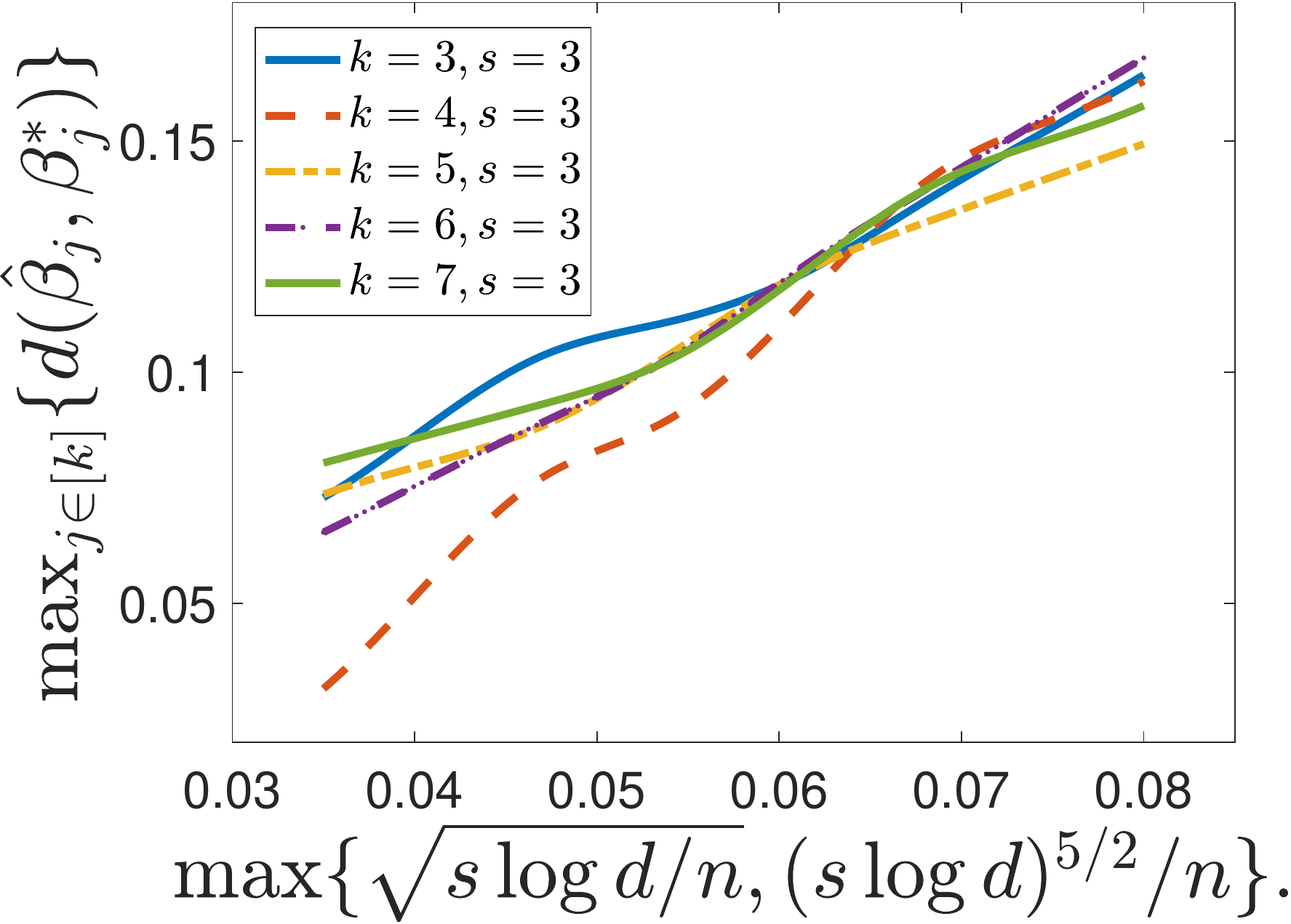}
		 &\hskip-10pt  \includegraphics[width=0.35\textwidth]{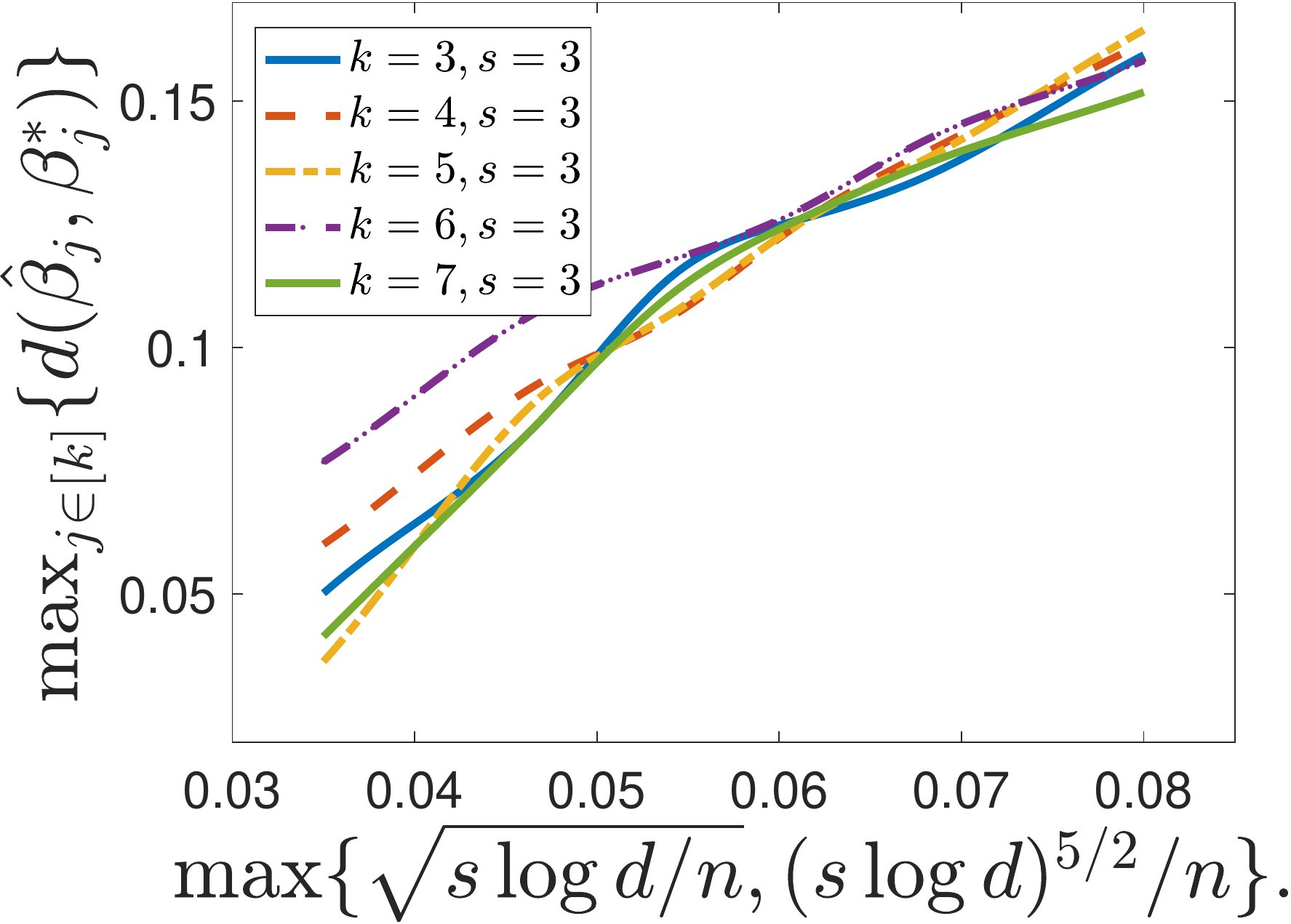} \\
		 \hskip-10pt \includegraphics[width=0.35\textwidth]{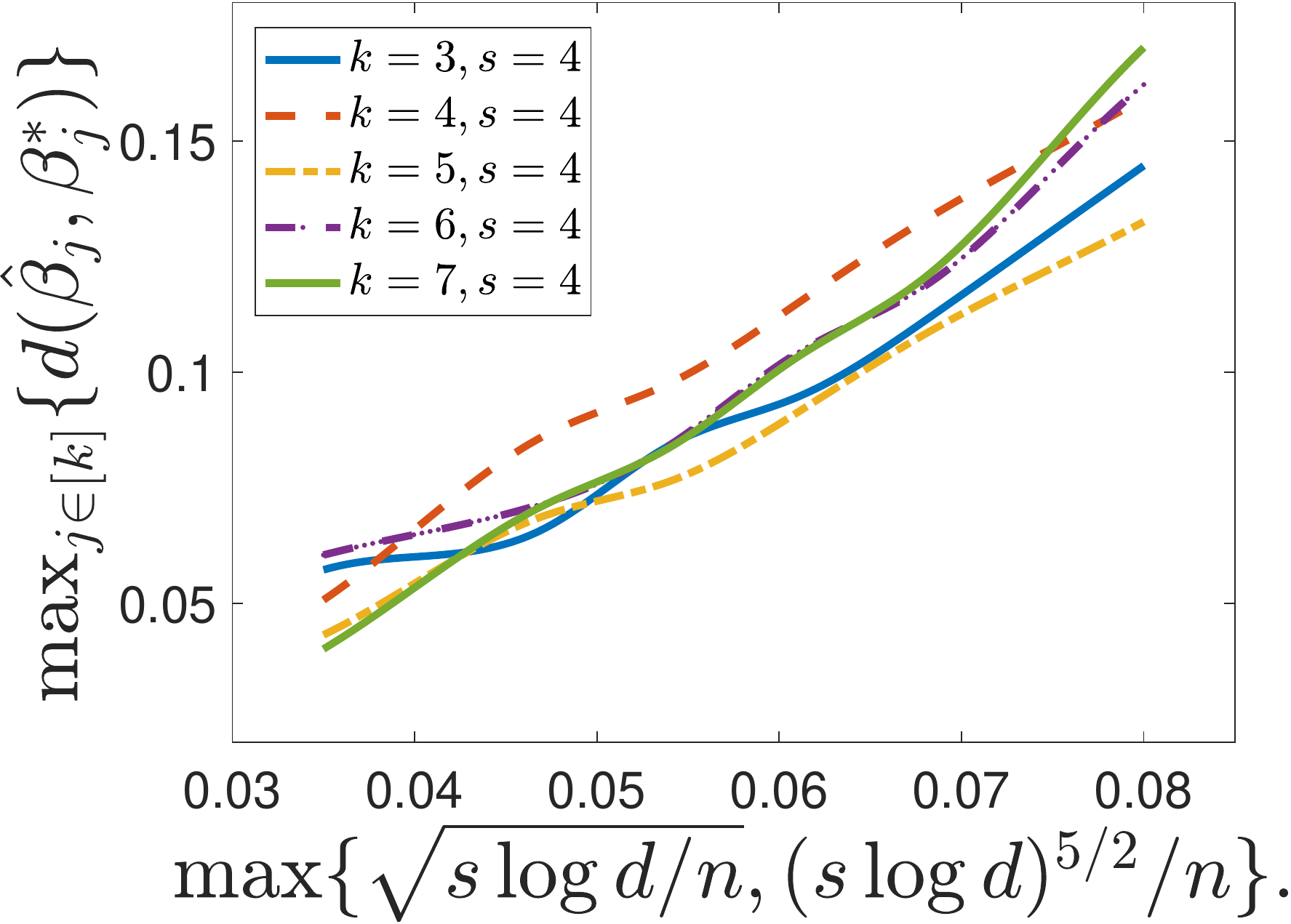}
		&
		\hskip-10pt \includegraphics[width=0.35\textwidth]{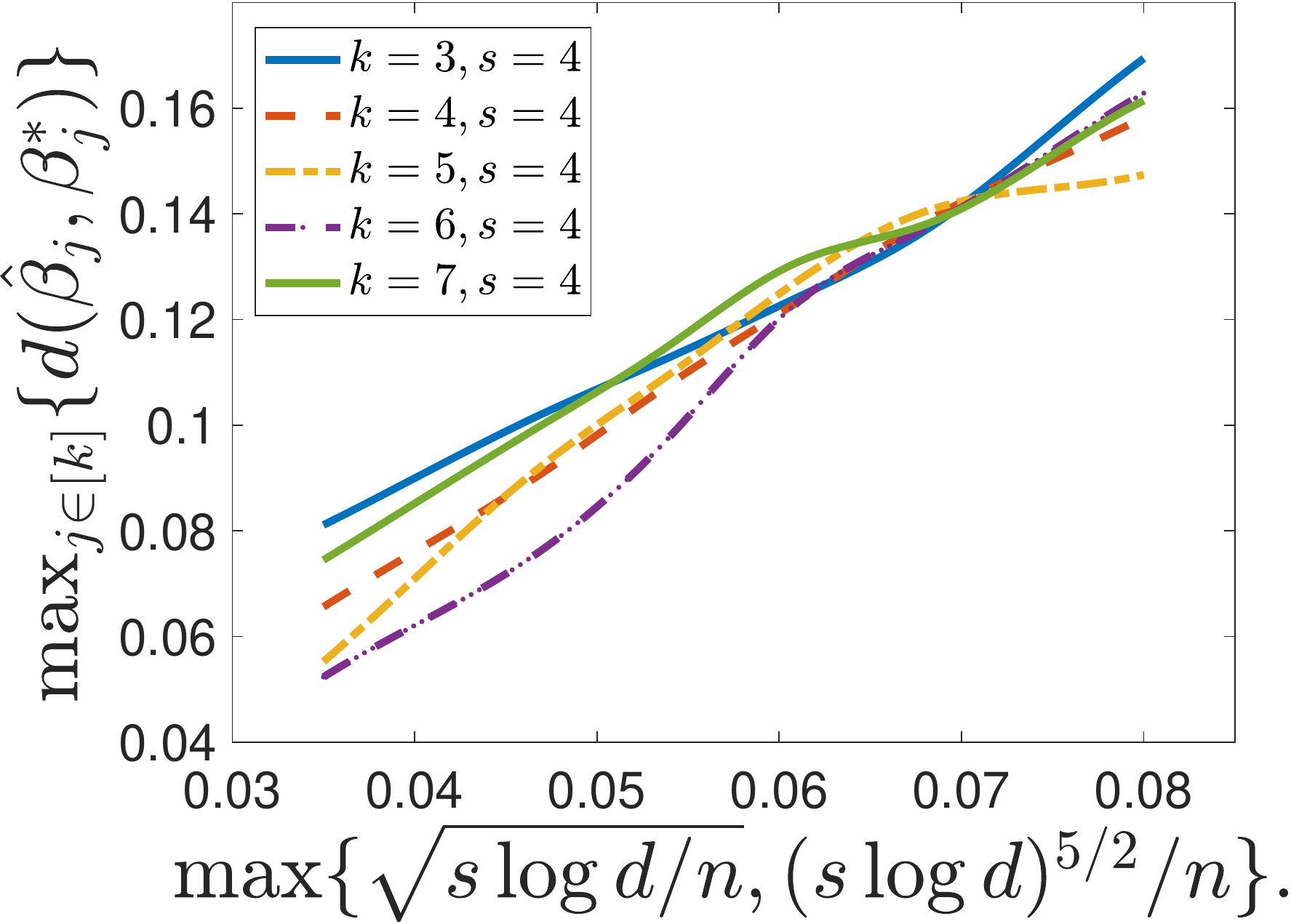}
		 &\hskip-10pt  \includegraphics[width=0.35\textwidth]{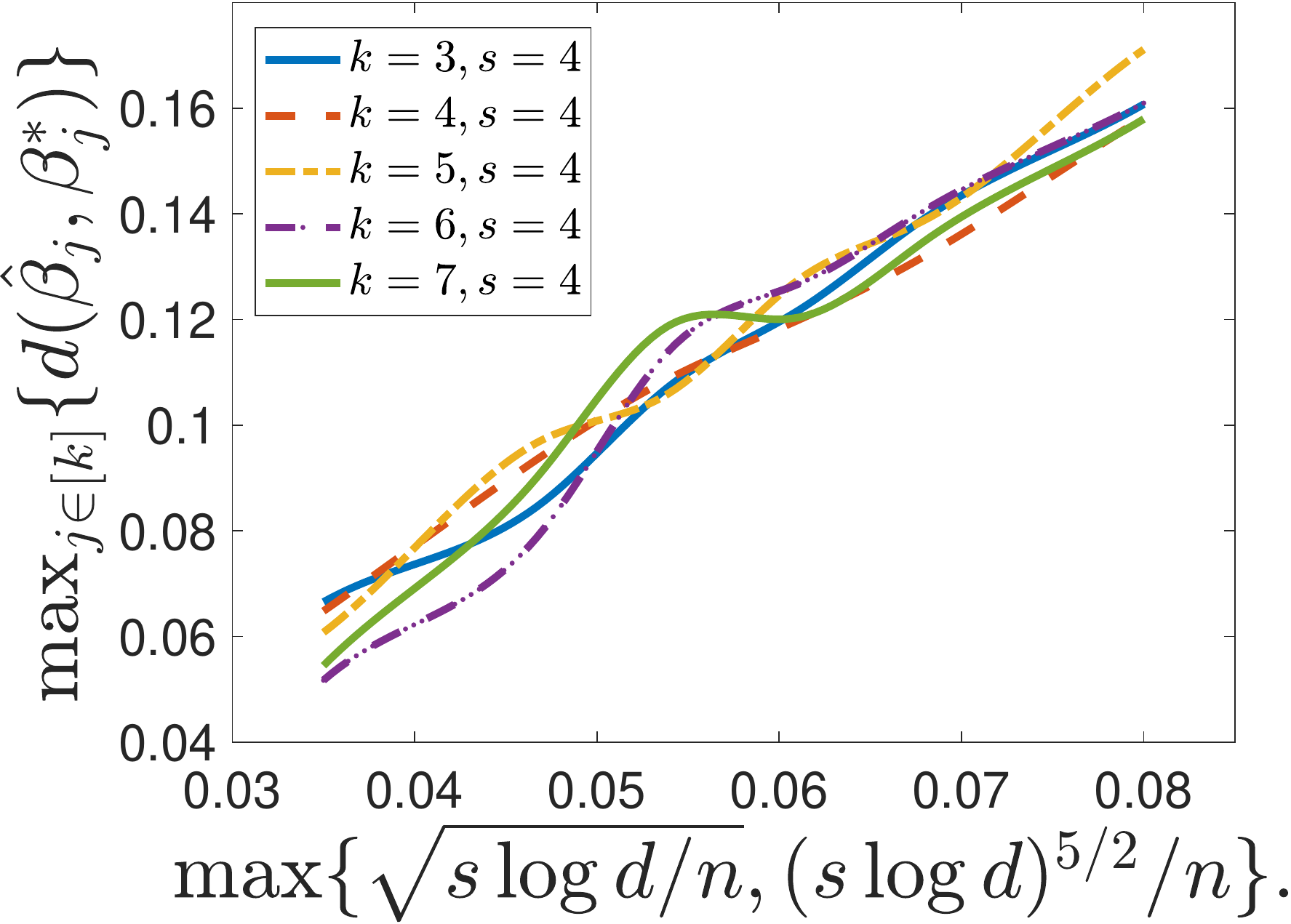} \\
		 \hskip-10pt \includegraphics[width=0.35\textwidth]{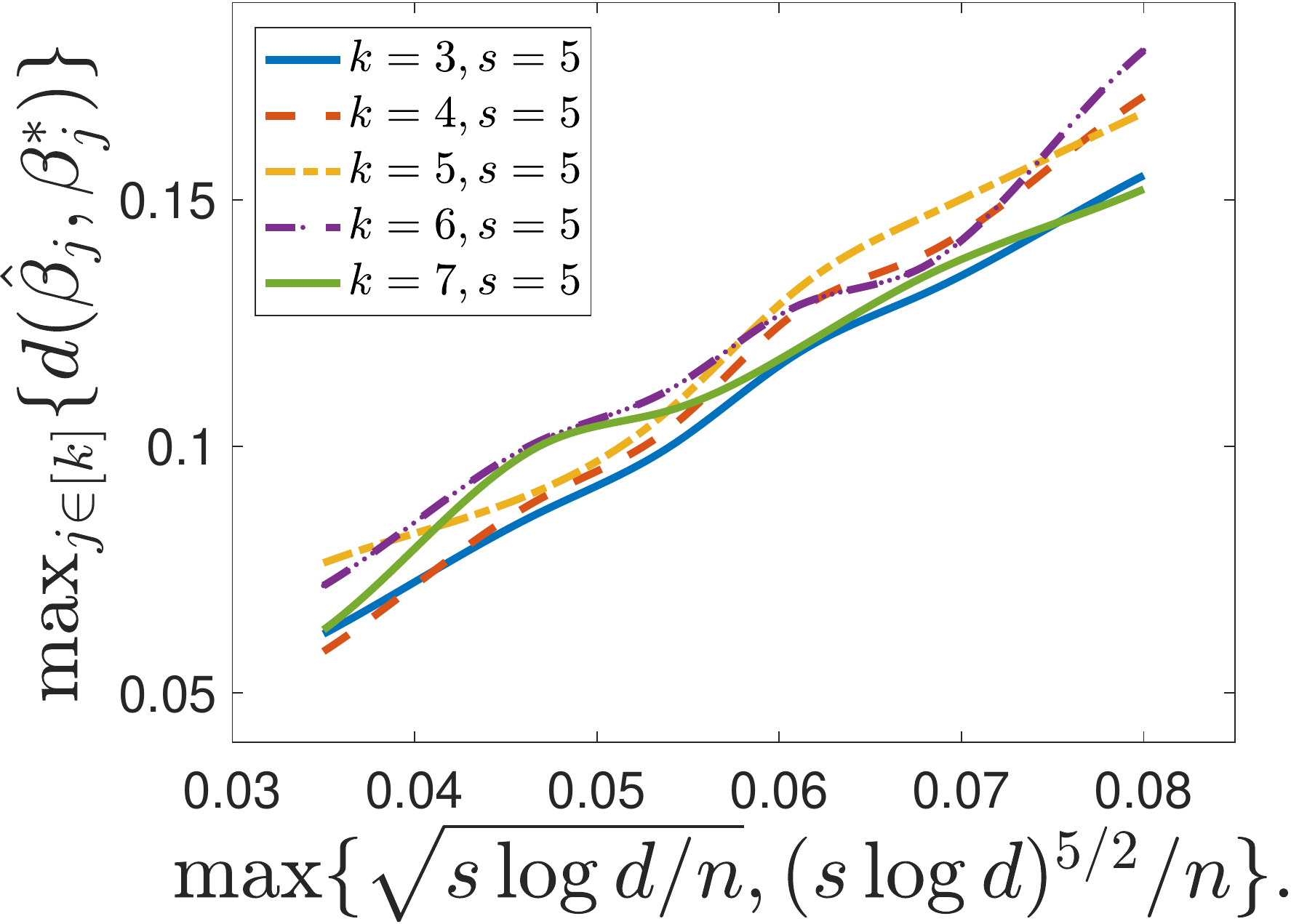}
		&
		\hskip-10pt \includegraphics[width=0.35\textwidth]{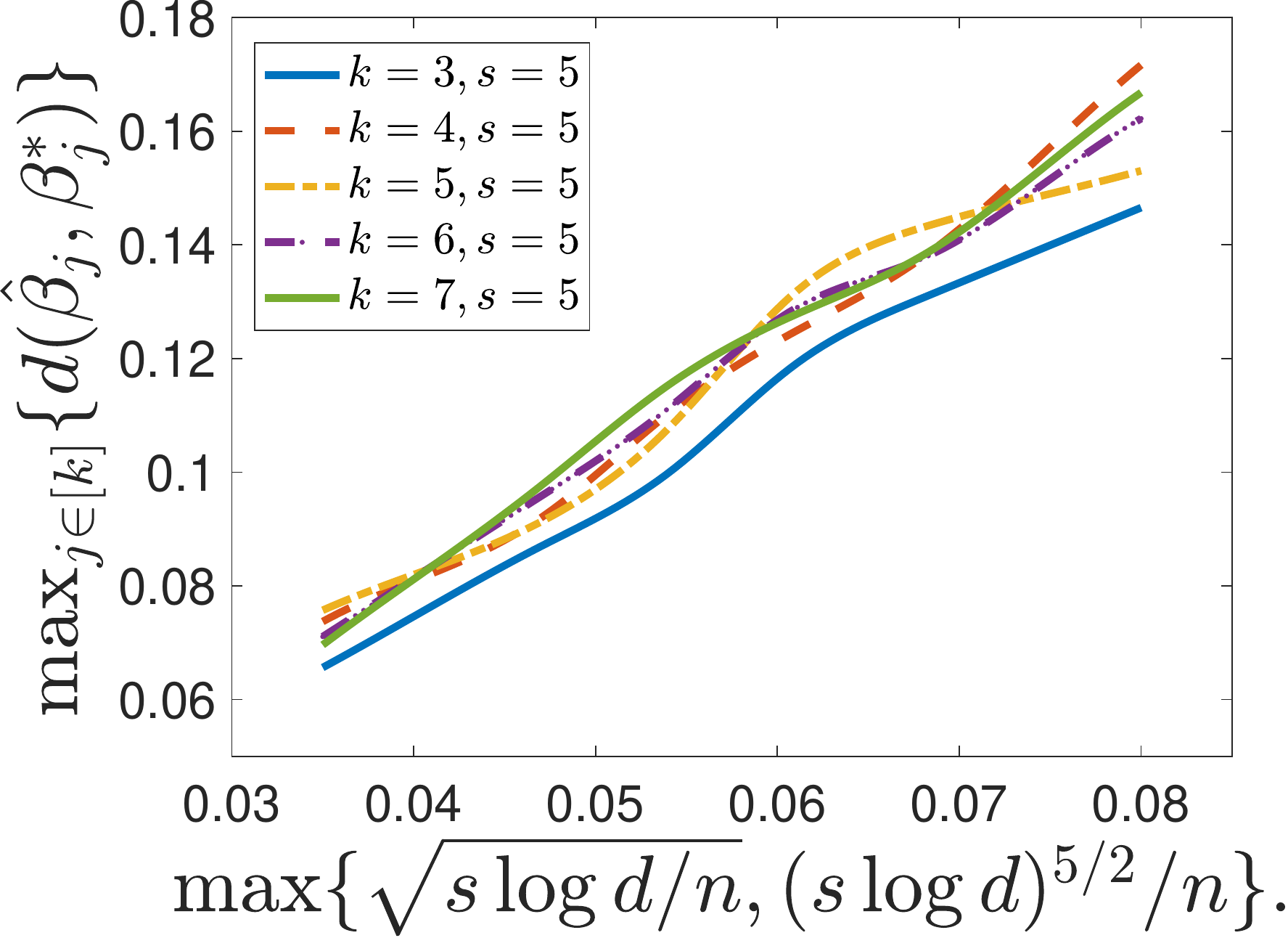}
		 &\hskip-10pt  \includegraphics[width=0.35\textwidth]{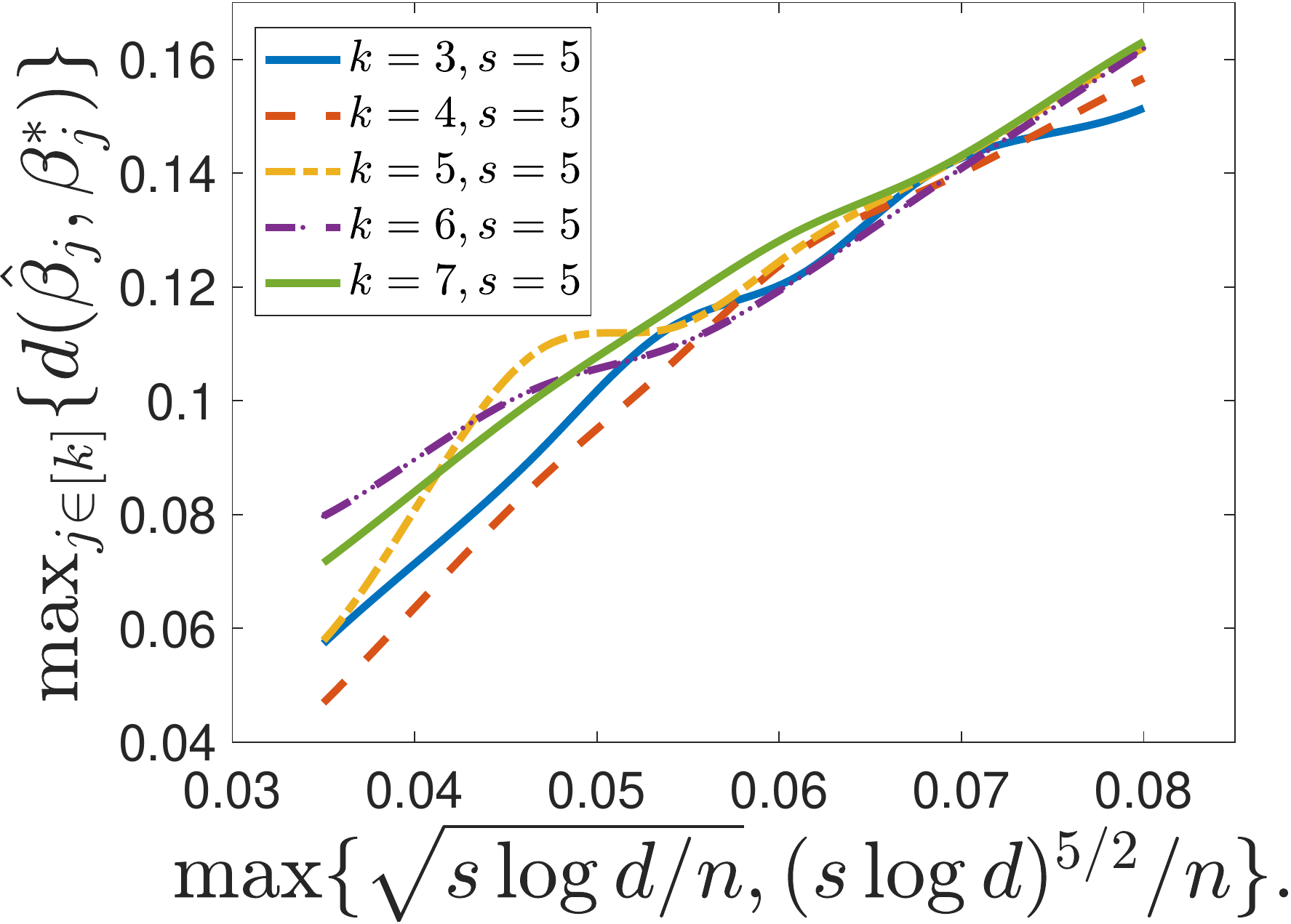} \\
		(a) $h_1(u)  = u^3 + 10 \cdot  \exp( - u^2)  $ & (b) $ h_2(u)  = u^3 + 5 \cdot \sin (2\cdot u^2)$, & (c) $ h_3(u)  =u^3 + 10 \cdot \tanh(u^2)$  	
	\end{tabular}\\
	\caption{\small \em Plots of the  estimation error    $\max_{j\in [k]} \{ d(\hat \beta_j , \beta_j^*) \}$ against the inverse signal strength $\max \{ \sqrt{s\log d/n} , ( s \log d)^{5/2} / n \}$ are  the output of Algorithm \ref{alg:ocpower} with the input moment tensor $\hat M$ defined in \eqref{eq:moment_tensor1}.   The  link function is  one of $h_1$, $h_2$, and $h_3$ in \eqref{eq:exper_links}. Moreover, we set  $d = 100 $, $s \in \{3,4,5\}$, $k \in \{3, \ldots, 7\}$,  and let $n$ vary. We generate each   figure   based on $100$ independent trials for each $(n, d, s, k)$.  }
	\label{fig:simulation2}
\end{figure}


\section{Related Work} \label{sec:rwork}

To the best of our knowledge, there is no related work on the general DAIM considered in this work apart from linear correspondence retrieval~\citep{andoni2017correspondence}. Mixture models are a popular class of models in the literature to handle heterogeneity, with applications to regression~\citep{chaganty2013spectral, zhong2016mixed}, classification~\citep{jacobs1991adaptive, sun2014learning, sedghi2016provable} and clustering~\citep{mclachlan2004finite, verzelen2017detection,zhao2017fast}. In terms of estimation, most of the above works focus on the parametric and/or low-dimensional setting. In comparison, handling heterogeneity in a high-dimensional and completely nonparametric setting is much more challenging. Recently, a mixture of nonparametric regression model was analyzed by~\cite{huang2013nonparametric}. A related technique of modal regression was also analyzed by~\cite{chen2016nonparametric}. Both methods are based on kernel smoothing techniques. While being completely nonparametric, they also suffer from the curse of dimensionality and are not applicable in the overcomplete setting. An interesting comprise is offered by the semiparametric single index model that we study in this work and described in detail in \S\ref{sec:models}. Such a model is popular in the econometrics literature and in particular allows us to deal with the overcomplete high-dimensional setting efficiently. The recent work of~\cite{xiang2016mixture} also proposed a similar model in the low-dimensional setting. But they do not address the statistical and computational issues associated with estimation in such models and their estimation procedure is not applicable in the overcomplete high-dimensional setting we consider. Furthermore, recall that our estimators are based on decomposing certain higher-order moment tensors. In the recent past, several works have proposed the use of tensor decomposition techniques for estimation in several parametric latent variable models apart from the ones cited above. Apart from mixture models for prediction, such techniques have also been used for estimating means of a Gaussian mixture model~~\citep{hsu2013learning,anandkumar2014tensor}, for estimating community membership in mixed membership models~\citep{, anandkumar2013tensor}, estimating the parametric components in generalized linear model~\citep{sedghi2016provable} and hidden-layer neural networks~\citep{zhong2017recovery,mondelli2018connection} and tensor sketching~\citep{hao2018sparse}. While being related, those works essentially consider parametric latent variable models in low-dimensional settings and are not suitable to the cases with unknown links or model misspecification that we consider. Furthermore, several works have also concentrated on coming up with faster algorithm for tensor decomposition in general (see for example the recent work of ~\cite{ge2017optimization} and the references there in). We remark that any such algorithmic advances could be directly applicable to the model that we consider.



\section{Discussion}\label{sec:ending}
In this work, we propose moment-based estimators for estimating the parametric components of overcomplete discordant additive index models, in both low and high-dimensional setting. Our models are motivated by sampling and heterogeneity issues common in high-dimensional big data paradigm. Our estimators are based on using tensor power method to decompose certain higher-order moment tensors. We establish statistical rate of convergence for our estimators. Numerical results are provided to corroborate the theoretical results. 

We conclude the paper with a discussion of two future directions. While the main focus of the paper was on estimating the parametric components of the model in~\eqref{eq:discoaim}, one can use the estimated parametric components to obtain the nonparametric components as well as a second step, using the idea of errors in variable additive index model~\citep{fan1993nonparametric, liang1999estimation, carroll2006measurement}. Next, in this paper we mainly considered the case of Gaussian covariates. Using the score based version of Stein's identity~\citep{stein2004use}, assuming knowledge of the density of $X$, our methods could be extended to non-Gaussian covariates. A potential issue is the score function might be heavy-tailed. Recently~\cite{yang2017estimating} proposed a method for dealing with heavy-tailed score functions but it is not clear how to extend their approach for the overcomplete setting. Furthermore, using the zero-bias transformation version of Stein's identity and assuming a stringent structure on the parameter vectors $\beta_j^*$, similar estimation rates with sub-Gaussian covariates could be obtained. Relaxing the Gaussian assumption, without additional assumption on either the covariates or on the parameter vectors seems to be a much harder task that we plan to address in the near future.

\appendix{}
\section{Proofs of Main Results}\label{sec:aux}
In this section, we first provide the proofs of moment bounds, concentration bounds and net argument used to obtained the tensor concentration results. Before we proceed, we recall the definition of the $\psi_1$-norm and $\psi_2$-norm for a random variable $X$:
 \#\label{eq:psi_2}
 \| X \|_{\psi_1} = \sup _{p \geq 1} \bigl \{ p^{-1} \cdot  \bigl ( \EE | X| ^{p} \bigr )^{1/p} \bigr \}
 \qquad  \| X \|_{\psi_2} = \sup _{p \geq 1} \bigl \{ p^{-1/2} \cdot  \bigl ( \EE | X| ^{p} \bigr )^{1/p} \bigr \}.
\#
Such norms are closely associated with the notion of sub-Gaussian and sub-exponential random variables that are standard in the literature on high-dimensional statistics; we refer the reader to~\citep{vershynin2010introduction} for a detailed discussion on such random variables and associated results. 

\subsection{Moment Bounds}

\begin{lemma}\label{lemma:psi_1_norm}  Let $\{ A_j \}_{j\in [m]} \subseteq \RR$ be $m$ sub-exponential random variables.  Let  $\Phi =\max_{j\in[m]}  \| A_j \|_{\psi_1}  $, then   $m^{-1} \sum_{j\in [m]} A_j$ is a sub-exponential random variable with $\psi_1$-norm bounded by $\Phi$.
\end{lemma} 
\begin{proof}
For simplicity, let $B = m^{-1} \sum_{j \in [m]} A_j $.  
By the definition of $\psi_1$-norm, we bound  $\EE (| B|^p)$ for any $p \geq 1$. By the triangle inequality and the AM-GM inequality, we have 
\#
& \EE (| B|^p)  \leq  \frac{1}{m^p}  \sum_{j_1, j_2, \ldots, j_p}  \EE \bigl ( | A_{j_1} A_{j_2}   \cdots A_{j_p} | \bigr )\label{eq:some_inequality1} \\
 &\qquad \leq \frac{1}{m^p\cdot p}  \sum_{j_1, j_2, \ldots, j_p} \bigl [  \EE (  | A_{j_1}|  ^p ) + \ldots + \EE (  | A _{j_p}|  ^p ) \bigr ], \label{eq:some_inequality2}
\#
where \eqref{eq:some_inequality1} follows from the triangle inequality, and \eqref{eq:some_inequality2} follows from the GM-AM inequality. Moreover, since $\| A_j \|_{\psi_1} \leq \Psi$ for any $j \in [m]$, we have 
\#\label{eq:some_ineq3}
\EE ( |A_j | ^p) \leq  ( p \cdot \Phi)^p,~~\text{for any} ~~p \geq 1.
\#
Thus, combining \eqref{eq:some_inequality1}, \eqref{eq:some_inequality2}, and \eqref{eq:some_ineq3}, we obtain that 
$
p^{-1} \cdot   [ \EE  (| B|^p)  ]^{1/p}   \leq  \Phi
$
for any $p\geq 1$, i.e., $\| B\|_{\psi_1} \leq \Phi$. Therefore,  we conclude Lemma \ref{lemma:psi_1_norm}. 
\end{proof}

\subsection{Concentration Results} \label{sec:concentration_results}
Since the response $Y $ in the additive index model is  sub-exponential, we also need to consider concentration results involving sub-exponential random variables. 
For convenience of the reader,  we  first briefly recall the present a result on the concentration of polynomial functions of  sub-Gaussian random vectors in  \cite{adamczak2015concentration}, which is applied  in Lemma \ref{lemma:sum_sub_exponential} below. Recall that, $X$ is sub-Gaussian if its $\psi_2$-norm is bounded. 

Moreover, in the following, we introduce a norm for tensors, which will  be used in the concentration results. 
 Let $\ell \in \NN_{+}  $ be a positive integer. We denote by $\cP_{\ell}$ the set of its partitions of $[ \ell]$  into non-empty and non-intersecting disjoint sets. Moreover, let $A = (a_{\ib})_{\ib \in [n]^\ell}$  be a tensor of order-$\ell$, whose entries are of the form 
 \$  a_{\ib} = a_{i_1, i_2, \ldots, i_{\ell}},~~\text{where}~~{\ib} = ( i_1, i_2, \ldots, i_{\ell} ).\$ 
 Finally, let  $\cJ = \{J_1, \ldots, J_k\} \in \cP_{\ell}$ be a fixed partition of $[\ell]$, where $J_j \subseteq [ \ell]$ for each $ j \in [k]$.  Let $|\cJ|$  denote the  cardinality of the   $\cJ$, which is equal to $k$.  We define a norm $\| \cdot \|_{\cJ}$ by 
\#\label{eq:CJ_norm}
\|A \|_{\cJ} = \sup\Bigg   \{\sum_{\ib \in [n]^\ell} a_{\ib} \prod_{j = 1}^k \bx^{(j )}_{{\ib}_{J_j}} : \bigl \| \bx^{(j)}   \bigr \| _2 \leq 1, \bx^{(j) } \in \RR^{n^{|J_j|}},  1 \leq j  \leq k \Bigg\},
\#
where we write  $\ib_I = (i_k)_{k \in I}$ for any $I \subseteq [ \ell]$ and the supremum is taken over all possible  $k$ vectors $\{ \bx^{(1)}, \ldots, \bx^{(k)}\} $. Here each $ \bx^{(j )} $ in \eqref{eq:CJ_norm} is a vector of dimension $n^{| J_j|}$ with Euclidean norm no more than one. Suppose $J_j = \{ t_1, t_2, \ldots, t_\alpha \} \subseteq [ \ell] $, then the $\ib_{J_j}$-th entry of $\bx^{(j)}$ is 
\$
\bx^{(j)}_{{\ib}_{J_j}}  = \bx^{(j)}_{i_{t_1}, i_{t_2}, \ldots, i_{t_{\alpha}} }.
\$    

 The norm defined in \eqref{eq:CJ_norm} is a generalization of some commonly seen  vector  and matrix norms. For example, when $\ell = 1$, \eqref{eq:CJ_norm} is reduced to the Euclidean norm of vectors in $\RR^n$.
 In addition,  let $A \in \RR^{n \times n}$ be a matrix, then $\cJ$ is a partition of $\{ 1, 2\}$, which implies that $\cJ $ is either $    \{ [2] \}   $ or $  \{ \{ 1 \}, \{ 2 \} \}.$  By the definition in \eqref{eq:CJ_norm}, we have 
 \$
 \| A \|_{\{[2] \}} = \sup \biggl \{  \sum_{i,j  \in [n]} a_{ij} x_{ij} \colon \sum_{ij \in [n] } x_{ij}^2 \leq 1 \biggr  \} = \| A \| _{F},
 \$
 which recovers the matrix Frobenius norm. Moreover, when $\cJ =  \{ \{ 1 \}, \{ 2 \} \}$, we have 
 \$
  \| A \|_{\{ 1\},\{ 2 \}} = \sup  \biggl \{  \sum_{i,j  \in [n]} a_{ij} x_{i} y_j  \colon \sum_{i \in [n] } x_{i} ^2 \leq 1 , \sum _{j \in [n] } y_j^2 \leq 1 \biggr  \} = \| A \| _{\oper},
 \$
 which is the operator norm of $A$.   Based on the norm $\| \cdot \|_{\cJ}$ defined in \eqref{eq:CJ_norm}, we introduce a concentration result for polynomials  of sub-Gaussian random vectors, which is  a simplified   version of Theorem 1.4 in  \cite{adamczak2015concentration}.

\begin{theorem}[Theorem 1.4 \cite{adamczak2015concentration}] \label{thm:subgauss_poly} Let $X = (X_1, \ldots, X_n) \in \RR^n$ be a random vector with independent components. Moreover, we assume that  such that for all $i \in  [n]$, we have $\|X_i\|_{\psi_2} \leq \Upsilon$, where the $\psi_2$-norm is defined in \eqref{eq:psi_2}. Then for every polynomial $F : \RR^n  \rightarrow \RR$ of degree $L$, we have:
$$
\PP\bigl ( \bigl | F(X) - \EE [ F(X) ]  \bigr |\geq t \bigr ) \leq 2 \exp\bigl [  - K_{L} \cdot \eta_{F} (t) \bigr ],
$$
where the univariate function $\eta_{F}(t) $ is given by
\#\label{eq:error_growing_fun}
\eta_{F} (t) = \min _{1\leq \ell \leq L } \min_{\cJ \in \cP_\ell}  \biggl [ \frac{t}{\Upsilon^\ell     \cdot  \bigl \|\EE [  D^\ell F(X)] \bigr \|_{\cJ} } \biggr ]^{2/ |\cJ|}.
\#
Here $\| \cdot \| _{\cJ}$ is defined in \eqref{eq:CJ_norm}, $\| \cdot \|_{p}$ is the $\ell_p$-norm of a random variable, and $D^\ell  F (\cdot ) $ is the $\ell$-th derivative of $F$, which is takes values in the $\ell$-th order  tensors.
\end{theorem}

Based on this theorem, we  are ready to introduce a concentration inequality for the product of two sub-exponential random variables. This inequality might be of independent interest.
\begin{lemma} 
    \label{lemma:sum_sub_exponential}
    Let $\{ (X_i, Y_i) \}_{i \in [n]}$   be $n$ independent copies of random variables $X$ and $Y$. We assume that $X$ is a sub-Gaussian random variable with $\| X \|_{\psi_2} \leq \Upsilon_1$, and $Y$ is a sub-exponential random variable with   $\| Y\|_{\psi_1} \leq \Upsilon_2$ for some constants $\Upsilon_1$ and $\Upsilon_2$. Here the $\psi_1$- and $\psi_2$-norms are defined in \eqref{eq:psi_2}.  
    Then for any $t \geq K \cdot \max\{  \Upsilon_1 ^3, \Upsilon_1 \}   \cdot \Upsilon_2$, we have
    \#
        & \PP \biggl \{ \bigg|\frac{1}{n}\sum_{i=1}^n \big[X_i^3 \cdot Y_i - \EE(X^3 Y)\big]\bigg|  \geq t \biggr \} \notag \\
         & \qquad  \leq 4 \exp \biggl \{  - K_1 \cdot \min \biggl[ \biggl ( \frac{t}{\sqrt{n} \Upsilon_1^3 \cdot \Upsilon_2} \biggr)^{2}, \biggl ( \frac{t}{\Upsilon_1^3 \cdot \Upsilon_2} \biggr )^{2/5} \biggr] \biggr \},\label{eq:third_order_concentration} \\
        &  \PP \biggl \{ \bigg|\frac{1}{n}\sum_{i=1}^n \big[X_i  \cdot Y_i - \EE(X  Y)\big]\bigg|  \geq t \biggr \} \notag \\
        & \qquad  \leq 4 \exp \biggl \{  - K_2\cdot \min \biggl[ \biggl ( \frac{t}{\sqrt{n} \Upsilon_1  \cdot \Upsilon_2} \biggr)^{2}, \biggl ( \frac{t}{\Upsilon_1 \cdot \Upsilon_2} \biggr )^{2/3} \biggr] \biggr \}. 
         \label{eq:third_order_concentration2}
    \#
   Here $K$,  $K_1$ and $K_2$ are absolute constants.
\end{lemma}

\begin{proof} 
We first establish \eqref{eq:third_order_concentration}.
For any $i \in [n]$, we  define a random variable $A_i^+$ as the positive part of $X_i^3 \cdot Y_i$ and let $A_i^{-}$ be the negative part. That is, we let  $A_i^+ = X_i^3 Y_i \cdot \ind  \{  X_i^3 \cdot Y_i \geq 0\}$ and   $A_i^- = A_i^+ -  X_i^3 \cdot Y_i$.  By these definitions, we have  $X_i^3 \cdot Y_i = A_i^+ - A_i^-$.
In the following, we establish upper bounds for 
\#\label{eq:k_upsilon}
\PP \biggl ( \biggl | \frac{1}{n} \sum_{i=1}^n  A_i^+ - \EE(A_i^+) \bigg| \geq t \biggr )\quad \text{and} \quad \PP \biggl ( \biggl | \frac{1}{n} \sum_{i=1}^n  A_i^- - \EE(A_i^-) \bigg| \geq t \biggr )
\#
for any $t \geq K\cdot  \Upsilon_1^3 \cdot \Upsilon_2$, where $K > 0$ is some absolute constant that will be specified later.
Note that, by symmetry, it  suffices to bound the first term in \eqref{eq:k_upsilon}.
Our proof utilizes the concentration inequality for polynomials of sub-Gaussian random variables, which is  given in Theorem \ref{thm:subgauss_poly}. 
To proceed, we first define  $n$ random variables $Z_i = \eta_i|A_i^+|^{1/5}$ for   $i \in [n]$, where $\{ \eta_i\}_{i\in [n]} $ are $n$ independent  Rademacher random variables.  We show that $\{ Z_i\}_{i\in [n]}$ are  $n$  i.i.d. sub-Gaussian random variables. Notice that by definition, we have    $| Z_i | =  | A_i^{+} | ^{1/5} \leq | X_i^3 \cdot Y_i | ^{1/5}.$ By H\"older's  inequality, for any integer $p \geq 1$, we have 
\#\label{eq:holder_psi1}
\EE (|Z_i|^p) \leq\EE\big ( |X_i|^{3 p/5} \cdot |Y_i|^{p/5}\big ) \leq  \bigl [ \EE (  |X_i|^{3p/2 } ) \bigr ]^{2/5} \cdot \bigl [ \EE (  |Y_i|^{p/3}) \bigr ] ^{3/5}.
\#
In addition, by the definition of the $\psi_2$-norm, we have 
\#\label{eq:use_x_bound}
 \EE (  |X_i|^{3 p/2} )   \leq  \bigl ( \sqrt{3p/2} \cdot \| X  \|_{\psi_2} \bigr ) ^{3p / 2}  \leq  ( 3p /2)^{3p /4 }\cdot \Upsilon_1 ^{3p/2},
\#
where we use the fact that $\| X  \|_{\psi_2}  \leq \Upsilon_1$. 
Similarly, for $Y_i$, by the definition of the $\psi_1$-norm, we have 
\#\label{eq:use_y_bound}
\EE \bigl ( | Y_i |^{p/3}  \bigr ) \leq  \bigl (p /3  \cdot \| Y_i \|_ {\psi_1}  \bigr )^{p/3 } \leq ( p / 3) ^{p/3 } \cdot \Upsilon_2 ^{p / 3}, 
\#
where $\Upsilon_2$ is an upper bound for $\| Y \|_{\psi_1}$. 
Combining \eqref{eq:holder_psi1}, \eqref{eq:use_x_bound}, and \eqref{eq:use_y_bound},  we obtain
\#\label{eq:z_bound}
\bigl [ \EE ( |Z_i|^p) \bigr ] ^{1/p} & \leq \bigl [ ( 3p /2)^{3p /10  } \cdot  ( p /3)^{p/5} \cdot \Upsilon_1 ^{3p /5} \cdot \Upsilon_2 ^{p/5} \bigr ] ^{1/p} \\ & \leq \sqrt{ 3p/2} \cdot  \Upsilon_1^{3/5}\cdot \Upsilon_2^{1/5} . \nonumber
\#
Hence, by the  definition of $\psi_2$-norm and \eqref{eq:z_bound}, we have
\#\label{eq:z_psi2}
\|Z_i\|_{\psi_2}= \sup_{p\ge 1}\bigl \{  p^{-1/2}\bigl [ \EE ( |Z_i|^p) \bigr ] ^{1/p} \bigr \}\leq\sqrt{ 3/2} \cdot  \Upsilon_1^{3/5}\cdot \Upsilon_2^{1/5},
\# 
which implies that $Z_i$ is a sub-Gaussian random variable with $\psi_2$-norm bounded by a constant depending on $\Upsilon_1$ and $\Upsilon_2$.\\
\noindent In the rest of the proof of \eqref{eq:third_order_concentration}, we establish a concentration inequality for $\{ Z_i^5\}_{i \in [n]}$. To apply Theorem \ref{thm:subgauss_poly}, we let  $f(u) = u^5$ and  define $F \colon \RR^n \rightarrow \RR$ by $F(z ) = \sum_{i = 1}^n f(z _i)$. 
Then by definition, the high-order derivatives of $F$ are diagonal tensors whose only nonzero entries are diagonal.  Specifically, for any $\ell \in [5]$ and any $i_1, i_2, \ldots, i_{\ell} \in [n]$, we have 
\$
\bigl [ D^\ell F(z ) \bigr ]_{i_1, \ldots, i_\ell}  =  \ind \{ i_1 = i_2 = \cdots = i_\ell \} \cdot f^{(\ell)}  ( z _{i_1}),
\$
where $z$ is  an arbitrary vector in $ \RR^n$. 
To simplify the notation,  for any   $a_1, \ldots, a_n \in \RR$, we denote  by 
$
\diag_\ell\{a_1,\ldots, a_n\}   
$
the $\ell$-th order  diagonal tensor with diagonal entries $a_1, a_2, \ldots, a_n$. Using this notation, for any $\ell \in [5]$, 
we can write 
 \#\label{eq:cj_norm0}
 D^\ell F(z ) = \diag_{\ell} \bigl [  f^{(\ell)}(z_1), \ldots, f^{(\ell)}(z _n)\bigr ].
\#
Moreover, for diagonal tensors, the norm  $\| \cdot \|_{\cJ}$ defined in \eqref{eq:CJ_norm}  have simple forms.
For any $a \in \RR^n$ and any  integer $\ell \geq 1 $, we have 
\#\label{eq:cj_norm1}
\|\diag_\ell\{a _1,\ldots, a _n\}\|_{\cJ} = \ind  \{ |\cJ| = 1 \}  \cdot \| a \|_2  + \ind \{ |\cJ| \geq 2 \} \cdot  \|a \|_{\max}.
\#
In addition, note that  $| f^{(\ell)}(u)|  \leq 5  ! \cdot  |u| ^{5-\ell}$ for any $u \in \RR$. For i.i.d. random variables $\{ Z_i \}_{i\in [n]}$, we define  $Z = (Z_1, \ldots, Z_n)^{\top}$. Combining \eqref{eq:cj_norm0} and \eqref{eq:cj_norm1},  we obtain 
\#\label{eq:cj_norm2}
\bigl \|\EE [ D^\ell F(Z) ] \bigr \|_{\cJ} & = \ind \{ |\cJ| = 1 \} \cdot \sqrt{n} \cdot  \bigl | \EE   [ f^{(\ell)} (Z_1 ) ]  \bigr | + \ind \{ |\cJ| \geq 2 \} \bigl | \EE [  f^{(\ell)} (Z_1) ]  \bigr |. 
\# 
Now we are ready to apply Theorem \ref{thm:subgauss_poly}. The  function $\eta_{F} (t) $ defined in \eqref{eq:error_growing_fun} now becomes
\#
\label{eq:compare_ell_cases}
\eta_{F} (t) = \min_{\ell \in [5] }\min_{\cJ \in \cP_{\ell} } \Bigl \{ t \cdot \| Z_1 \|_{\psi_2}^{- \ell}    \cdot  \bigl \|\EE [ D^\ell F(Z)] \bigr \|_{\cJ} ^{-1}  \Bigr \}^{2/ |\cJ|}.
\#
Note that the tail probability in Theorem \ref{thm:subgauss_poly} is equal to $\exp[- K \cdot \eta_{F}(t) ]$, where $K$ is an absolute constant. To upper bound this term, in the sequel, we establish an lower  bound for $\eta_{F}(t)$.
We first establish an  upper bound for $\bigl \|\EE [ D^\ell F(Z) ] \bigr \|_{\cJ}  $. Since $f(u) = u^5$, we have  $| f^{\ell} (u)| \leq 5! \cdot | u|^{5 - \ell}$. Thus, by \eqref{eq:cj_norm2} we have
\$
\bigl \|\EE [ D^\ell F(Z) ] \bigr \|_{\cJ}    \leq 5! \cdot  \bigl (\ind  \{ |\cJ| = 1 \} \cdot \sqrt{n} + \ind \{ |\cJ| \geq 2 \} \bigr ) \cdot   \EE  \bigl ( |Z_1|^{5 - \ell} \bigr )     .
\$
Note that $Z_1$ is a sub-Gaussian random variable. By the definition of $\psi_2$-norm in \eqref{eq:psi_2}, we have $\EE (|Z_1|^k)  \leq (\sqrt{k})^{k} \cdot \|Z_1\|^{k}_{\psi_2}$ for any $k \geq 0$, where we follow  the convention by letting $0^0 =1$.
Therefore, by \eqref{eq:cj_norm2} we have 
\begin{align} \label{eq:cj_norm_final}
& \bigl \|\EE [ D^\ell F(Z)] \bigr\|_{\cJ}  \notag \\
& \quad \leq 25  \cdot \big(\sqrt{5-\ell} \big)^{5-\ell}\cdot   \bigl (\ind  \{ |\cJ| = 1 \} \cdot \sqrt{n} + \ind \{ |\cJ| \geq 2 \} \bigr ) \cdot \|Z_1\|^{5- \ell}_{\psi_2},
\end{align}
for any $\ell \in [5 ]$. We denote $C_{\ell} = 25 \cdot (5 - \ell)^{(5 - \ell)/2}$ to simplify the notation, which is an absolute constant. Now we combine \eqref{eq:compare_ell_cases} and \eqref{eq:cj_norm_final} to obtain
\$
\| Z_1 \|_{\psi_2}^{\ell}    \cdot  \bigl \|\EE [ D^\ell F(Z)] \bigr \|_{\cJ } \leq C_{\ell} \cdot \bigl (\ind  \{ |\cJ| = 1 \} \cdot \sqrt{n} + \ind \{ |\cJ| \geq 2 \} \bigr ) \cdot \|Z_1\|^{5}_{\psi_2}. 
\$ 
Plugging this inequality in \eqref{eq:compare_ell_cases},
we have 
\$
\eta_{F} (t) \geq  \min_{\ell \in [5] } \Bigl \{  \min \Bigl [ \bigl ( t / \sqrt{n}  \cdot C_{\ell}^{-1} \| Z_1 \|_{\psi_2}^{- 5 }   \bigr)^ {2} ,  \min_{2 \leq |\cJ| \leq \ell} \bigl(   t \cdot C_{\ell}^{-1} \| Z_1 \|_{\psi_2}^{- 5}   \bigr )^{2/ |\cJ|} \Bigr ]\Bigr \}.
\$ 

\noindent When $t > C_{\ell} \cdot \| Z_1 \|_{\psi_2}^{5}$ for any $\ell \in \{2, \ldots, 5\}$, we simplify  
the above inequality~by   
\#\label{eq:lower_eta_fun2}
\eta_{F} (t) & \geq \min \Bigl \{ \min_{\ell \in [5] } \Bigl [ \bigl ( t / \sqrt{n}  \cdot C_{\ell}^{-1} \| Z_1 \|_{\psi_2}^{- 5 }   \bigr) ^{2} \Bigr ], \min_{2 \leq \ell \leq 5 } \Bigl [   \bigl(   t \cdot C_{\ell}^{-1} \| Z_1 \|_{\psi_2}^{- 5}   \bigr )^{2/ \ell }    \Bigr ]\Bigr \} \notag \\
& = \min \Bigl \{  \bigl ( t / \sqrt{n}  \cdot  \tilde C_{0}^{-1} \| Z_1 \|_{\psi_2}^{- 5 }   \bigr) ^{2} , \min_{2 \leq \ell \leq 5 } \Bigl [   \bigl(   t \cdot C_{\ell}^{-1} \| Z_1 \|_{\psi_2}^{- 5}   \bigr )^{2/ \ell }    \Bigr ]\Bigr \} , 
\#
where we define $\tilde C_0 = \max_{\ell \in [5] } C_{\ell}$.
Note that the $\psi_2$-norm of $Z_1$ is bounded in \eqref{eq:z_psi2}, which implies 
\#\label{eq:universal_psi_bound}
C_{\ell} \cdot \| Z_1 \|_{\psi_2} ^5  \leq K \cdot \Upsilon_1^3 \cdot \Upsilon_2 
\#
for all $\ell \in [5]$, 
where $K$ is some absolute constant. We denote  the term on the right-hand side of \eqref{eq:universal_psi_bound} by $K_{\Upsilon}$ for simplicity. 

Thus, combining \eqref{eq:lower_eta_fun2} and  \eqref{eq:universal_psi_bound}, we have  
\$
\eta_{F}(t)  \geq \min \bigl  \{ \bigl [ t / ( \sqrt{n} \cdot K_{\Upsilon} )  \bigr ] ^2 , \min_{2 \leq \ell \leq 5} (t /  K_{\Upsilon } )^{2/ \ell}  \bigr \}. 
\$
When $t \geq K_{\Upsilon} $, we have 
\#\label{eq:bound_eta_final}
\eta_{F}(t) \geq\min \bigl  \{ \bigl [ t / ( \sqrt{n} \cdot K_{\Upsilon} )  \bigr ] ^2 ,  (t /  K_{\Upsilon } )^{2/ 5}  \bigr \}. 
\#
Plugging \eqref{eq:bound_eta_final} in Theorem \ref{thm:subgauss_poly}, since $|Z_i|^5 = | A_{i}^{+} | = A_{i}^{+}$, when $t \geq 2 K_{\Upsilon}$,  we obtain
\# \label{eq:first_term_tail}
\PP\biggl [ \biggl | \sum_{i=1 }^n  A_i^{+}  - \EE( A_{i}^{+}) \biggr |\geq t/2  \biggr] &  \leq 2 \exp \bigl[ - C\cdot \eta_{F}(t/2) \bigr ] \notag  \\
& \leq 2 \exp  \Bigl ( - C \min \bigl  \{ \bigl [ t / ( 2 \sqrt{n} \cdot K_{\Upsilon} )  \bigr ] ^2 ,   [ t / (2 K_{\Upsilon } ) ] ^{2/ 5}  \bigr \}  \Bigr ),  
\#
where $C$ is an absolute constant that  does not rely on $\Upsilon_1$ and $\Upsilon_2$.
Similarly, for $A_{i}^{-}$, using the same analysis, we obtain that 
\#\label{eq:second_term_tail}
& \PP\biggl [ \biggl | \sum_{i=1 }^n  A_i^{-}  - \EE( A_{i}^{-}) \biggr |  \geq t/ 2 \biggr] \notag \\
&\quad  \leq 2 \exp  \Bigl ( - C \min \bigl  \{ \bigl [ t / ( 2 \sqrt{n} \cdot K_{\Upsilon} )  \bigr ] ^2 ,   [ t / (2 K_{\Upsilon } ) ] ^{2/ 5}  \bigr \}  \Bigr ).
\#
Note that $|a - b| \geq t$ implies that $|a| \geq t/2$ or $|b| \geq t/2$ for any $a, b \in \RR$ and $t > 0$. 
Combining \eqref{eq:first_term_tail} and \eqref{eq:second_term_tail},  we finally obtain
\# \label{eq:111_final}
&  \PP\biggl [ \biggl | \sum_{i=1 }^n  X_i ^3 \cdot Y   - \EE( X^3 \cdot Y) \biggr |\geq t   \biggr] \notag \\
  & \quad \leq \PP\biggl [ \biggl | \sum_{i=1 }^n  A_i^{+}  - \EE( A_{i}^{+}) \biggr |\geq t/2  \biggr]    +\PP\biggl [ \biggl | \sum_{i=1 }^n  A_i^{-}  - \EE( A_{i}^{-}) \biggr |\geq t/ 2 \biggr] \notag \\
&  \quad  \leq 4 \exp \biggl \{  - K_1 \cdot \min \biggl[ \biggl ( \frac{t}{\sqrt{n} \Upsilon_1^3 \cdot \Upsilon_2} \biggr)^{2}, \biggl ( \frac{t}{\Upsilon_1^3 \cdot \Upsilon_2} \biggr )^{2/5} \biggr] \biggr \}, 
\#
where $K_1 $ is an absolute constant. 
Note that, by \eqref{eq:universal_psi_bound},  here we require $t \geq K \cdot  \Upsilon_1^3 \cdot \Upsilon_2$ for some constant $K$. Thus, we establish \eqref{eq:third_order_concentration}.

To conclude the proof of this lemma,  it remains to show \eqref{eq:third_order_concentration2}. The proof is similar to the derivations above. Now, for any $i \in [n]$,  we define $B_{i}^{+}$ and $B_{i}^{-}$ as the positive and negative parts of $X_i \cdot Y_i$, respectively. Then we have $X_i \cdot Y_i = B_{i}^{+} - B_{i}^{-}$ by definition.  

We first establish a concentration inequality for $\{ B_{i}^+ \}_{i\in [n]}$. To utilize  Theorem \ref{thm:subgauss_poly}. 
To proceed, 
we first define    $Z_i = \eta_i|B_i^+|^{1/3}$ for   $i \in [n]$, where $\{ \eta_i\}_{i\in [n]} $ are $n$ independent  Rademacher random variables.  Then we show that $\{ Z_i\}_{i\in [n]}$ are  $n$  i.i.d. sub-Gaussian random variables.  For any integer $p \geq 1$,  H\"older's  inequality implies that 
\#\label{eq:holder_psi12}
\EE (|Z_i|^p) \leq\EE\big ( |X_i|^{p /3} \cdot |Y_i|^{p/3}\big ) \leq  \bigl [ \EE (  |X_i|^{p/2 } ) \bigr ]^{2/3} \cdot \bigl [ \EE (  |Y_i|^{ p }) \bigr ] ^{1/3}.
\#
Moreover, since $\| X \|_{\psi_2} \leq \Upsilon_1$ and $\| Y \|_{\psi_1} \leq \Upsilon_2$,  by the definition of the $\psi_2$- and $\psi_1$-norms, we have 
\#\label{eq:use_xy_bound}
 \EE (  |X_i|^{p/2 } )  & \leq  \bigl ( \sqrt{p/2 } \cdot \| X  \|_{\psi_2} \bigr ) ^{p/2}  \leq  (p/2)^{p/4 }\cdot \Upsilon_1 ^{p/2}, \\
 \EE  ( | Y_i |^{p}   ) &\leq  \bigl ( p  \cdot \| Y_i \|_ {\psi_1}  \bigr )^{ p } \leq  p^{p}\cdot \Upsilon_2 ^{p  }. \nonumber
\#
 Thus, 
combining \eqref{eq:holder_psi12} and \eqref{eq:use_xy_bound},  we obtain
\$
\bigl [ \EE ( |Z_i|^p) \bigr ] ^{1/p}  \leq \bigl [ ( p/2)^{p/6   } \cdot  p^{p/3} \cdot \Upsilon_1 ^{2p /3} \cdot \Upsilon_2 ^{p/3} \bigr ] ^{1/p} \leq \sqrt{ p} \cdot  \Upsilon_1^{1/3}\cdot \Upsilon_2^{1/3} ,
\$
which implies that 
$
\|Z_i\|_{\psi_2} \leq   \Upsilon_1^{1/3}\cdot \Upsilon_2^{1/3}.
$
Thus,  $Z_i$ is a sub-Gaussian random variable with $\psi_2$-norm bounded  by $\Upsilon_1^{1/3}\cdot \Upsilon_2^{1/3}$.

To prove \eqref{eq:third_order_concentration2}, we establish a concentration inequality for $\{ Z_i^3 \}_{i \in [n]}$. Similar to the pervious case, we  let  $f(u) = u^3$ and  define $F \colon \RR^n \rightarrow \RR$ by $F(z ) = \sum_{i = 1}^n f(z _i)$. 
By this construction, for any $\ell \in \{1,2,3\}$, 
the $\ell$-th order derivative of $F$ is given by
 \#\label{eq:cj_norm1}
 D^\ell F(z ) = \diag_{\ell} \bigl [  f^{(\ell)}(z_1), \ldots, f^{(\ell)}(z _n)\bigr ].
\#
Note that here $D^\ell F(z ) $ has the same form as \eqref{eq:cj_norm0}. Similar to  \eqref{eq:cj_norm_final},
we have 
\#\label{eq:some_useless_term}
&\|Z_1\|^{ \ell}_{\psi_2} \cdot  \bigl \|\EE [ D^\ell F(Z)] \bigr\|_{\cJ} \notag \\
& \quad   \leq 3!  \cdot \big(\sqrt{3-\ell} \big)^{3-\ell}\cdot   \bigl (\ind  \{ |\cJ| = 1 \} \cdot \sqrt{n} + \ind \{ |\cJ| \geq 2 \} \bigr ) \cdot \|Z_1\|^{3 }_{\psi_2}  \notag \\
& \quad  \leq K \cdot  \bigl (\ind  \{ |\cJ| = 1 \} \cdot \sqrt{n} + \ind \{ |\cJ| \geq 2 \} \bigr ) \cdot \Upsilon_1 \cdot \Upsilon_2,
\#
where $K$ is a constant that does not depend on $\Upsilon_1$ and $\Upsilon_2$.
Moreover, by \eqref{eq:some_useless_term}, for any $t$ satisfying $t \geq K\cdot \Upsilon_1 \cdot \Upsilon_2 $, 
the  function $\eta_{F} (t) $ defined in \eqref{eq:error_growing_fun} can be lower  bounded by
\#
\label{eq:bound_eta_F}
\eta_{F} (t) 
&   = \min_{\ell \in [3] }\min_{\cJ \in \cP_{\ell} } \Bigl \{ t \cdot \| Z_1 \|_{\psi_2}^{- \ell}    \cdot  \bigl \|\EE [ D^\ell F(Z)] \bigr \|_{\cJ} ^{-1}  \Bigr \}^{2/ |\cJ|} \notag \\
&  s \geq \min \bigl  \{ \bigl [ t  \big/ ( \sqrt{n} \cdot K  \Upsilon_1  \Upsilon_2  )  \bigr ] ^2 ,  \bigl [ t \big /  (  K   \Upsilon_1   \Upsilon_2  ) \bigr ]^{2/ 3}  \bigr \}.
\#
Note that $|Z_i|^3 = | B_{i}^{+} | = B_{i}^{+}$. 
Plugging \eqref{eq:bound_eta_F} in Theorem \ref{thm:subgauss_poly}, for any $t$ satisfying  $t \geq 2 K  \cdot \Upsilon_1 \cdot \Upsilon_2$,  we have
\# \label{eq:first_term_tail1}
& \PP\biggl [ \biggl | \sum_{i=1 }^n  B_i^{+}  - \EE( B_{i}^{+}) \biggr |  \geq t/2  \biggr] \leq 2 \exp \bigl[ - C\cdot \eta_{F}(t/2) \bigr ] \notag \\
& \quad  \leq 2 \exp  \Bigl ( - C \min \bigl  \{ \bigl [ t / ( 2 \sqrt{n} \cdot K  \Upsilon_1  \Upsilon_2 )  \bigr ] ^2 ,   [ t / (2 K  \Upsilon_1  \Upsilon_2 ) ] ^{2/ 3}  \bigr \}  \Bigr ),
\#
where $C$ is an absolute constant that is does not reply on $\Upsilon_1$ and $\Upsilon_2$.
Similarly,    using the same analysis, we obtain a similar concentration inequality for $\{ B_i^{-} \}_{i\in [n]}$:
\#\label{eq:second_term_tail1}
& \PP\biggl [ \biggl | \sum_{i=1 }^n  B_i^{-}  - \EE( B_{i}^{-}) \biggr |\geq t/ 2 \biggr] \notag \\
& \quad \leq 2 \exp  \Bigl ( - C \min \bigl  \{ \bigl [ t / ( 2 \sqrt{n} \cdot K  \Upsilon_1  \Upsilon_2 )  \bigr ] ^2 ,   [ t / (2 K  \Upsilon_1  \Upsilon_2 ) ] ^{2/ 3}  \bigr \}  \Bigr ).
\#
Therefore, combining \eqref{eq:first_term_tail1} and \eqref{eq:second_term_tail1},  we finally obtain that
\# \label{eq:111_final2}
& \PP\biggl [ \biggl | \sum_{i=1 }^n  X_i   \cdot Y   - \EE( X  \cdot Y) \biggr |\geq t   \biggr] \notag s  \\
 & \quad \leq  4    \exp  \Bigl ( - K_2 \cdot  \min \bigl  \{ \bigl [ t / ( 2 \sqrt{n} \cdot  \Upsilon_1  \Upsilon_2 )  \bigr ] ^2 ,   [ t / (2   \Upsilon_1  \Upsilon_2 ) ] ^{2/ 3}  \bigr \}  \Bigr )
\#
where $K_2 $ is an absolute constant.   Note that here we require $t \geq K \cdot \Upsilon_1 \Upsilon _2 $ for some absolute  constant $K > 0$. Therefore,  combining \eqref{eq:111_final} and \eqref{eq:111_final2}, we conclude the proof of Lemma~\ref{lemma:sum_sub_exponential}. 
\end{proof}

\subsection{Net Argument for Tensor Operator Norm}
In this section, we prove the $\epsilon$-net argument for tensor operator norm. 
The $\epsilon$-net argument is a standard technique for bounding the operator norm of matrices, whereas  its construction is  relatively more involved in the tensor case.  

In the sequel, for generality, we focus on $\ell$-th order  tensors in $\RR^d$. For any $\ell$-th order tensor $A\in \RR^{d\otimes \ell} $, the operator norm of $A$ is given by 
\#\label{eq:def_tensor_opernorm}
\| A \|_{\oper} = \sup\bigl \{  \bigl | A(u^{(1)} , \ldots, u^{(\ell)} ) \bigr | \colon  u^{(1)}, u^{(2)}, \ldots, u^{(\ell)} \in \cS^{d-1} \bigr \} , 
\#
where $\cS^{d-1}$ is the unit sphere in $\RR^d$. In addition, for any $r \in [d]$, we define  the $r$-sparse subset of $\cS^{d-1}$    as 
\#\label{eq:sparse_sphere}
\cS^{d-1} (r)  = \bigl \{ u\in \RR^d\colon \| u \|_2 =1, \| u \|_0 \leq r \bigr \}.
\#
Then the sparse tensor norm of $A$ is defined by 
\#\label{eq:def_sparse_tensor_norm}
\| A \|_{\oper, r} = \sup\bigl \{  \bigl | A(u^{(1)} , \ldots, u^{(\ell)} ) \bigr | \colon  u^{(1)}, u^{(2)}, \ldots, u^{(\ell)} \in \cS^{d-1} (r) \bigr \} . 
\#
Moreover, for any set $\cE \subseteq \RR^d$, we say $\cN$  is an $\epsilon$-net for $\cE$, if for any $u \in \cE$, there exists $v \in \cN$ such that $\| u - v \|_2 \leq  \epsilon$.   Then we are ready to present the $\epsilon$-argument for $\ell$-th order tensors, which is given in the following lemma.

    \begin{lemma} [$\epsilon$-net argument for tensors]\label{lem:net}
  \label{lemma:tensor_net}
   For any $\epsilon \in (0, 1/\ell)$, let $\cN_{\cS}(\epsilon)$   and $\cN_{\cS}(\epsilon,r )$ be the  $\epsilon$-nets of $\cS^{d-1}$ and $\cS^{d-1} (r)$, respectively. Then for any   $\ell$-th order tensor $A \in \RR^{d\otimes \ell}$, we have  
\#
   \| A \|_{\oper}    &  \leq   (1 - \ell\cdot  \epsilon   )^{-1} \cdot  \sup\bigl \{  \bigl | A(u^{(1)} , \ldots, u^{(\ell)} ) \bigr | \colon  \{  u^{(j )} \}_{j \in [\ell]} \subseteq  \cN_{\cS}(\epsilon) \bigr \}, \label{eq:tensor_net1}\\
  \| A \|_{\oper, r }  &  \leq    (1 - \ell\cdot  \epsilon   )^{-1} \cdot  \sup\bigl \{  \bigl | A(u^{(1)} , \ldots, u^{(\ell)} ) \bigr | \colon   \{  u^{(j )} \}_{j \in [\ell]} \subseteq  \cN_{\cS}(\epsilon, r) \bigr \}. \label{eq:tensor_net2}
 \#
 Moreover, when $A$ is a symmetric tensor, we further have 
 \#
 \| A \|_{\oper} & \leq (1 - \ell\cdot  \epsilon  )^{-1}    \cdot  \sup_{u \in \cN_{\cS}(\epsilon) }   \bigl | A(u, u, \ldots, u ) \bigr |, \label{eq:tensor_net11} \\
  \| A \|_{\oper,r } & \leq (1 - \ell\cdot  \epsilon  )^{-1}    \cdot  \sup_{u \in \cN_{\cS}(\epsilon, r) }   \bigl | A(u, u, \ldots, u ) \bigr |.\label{eq:tensor_net21}
  \#
\end{lemma}

Note that our Lemma \ref{lemma:tensor_net} covers the $\epsilon$-argument  for matrices by setting $\ell = 2$. In this case,  we have 
\$
\| A \|_{\oper} = \sup_{x,y \in \cS^{d} } x^\top A y  \leq (1- 2\epsilon )^{-1} \sup_{x, y \in \cN_{\cS}(\epsilon) } x^\top A y,
\$
which is the Lemma 5.4 in  \cite{vershynin2010introduction}. 

\begin{proof}[Proof of Lemma~\ref{lem:net}]
In the following, we first prove the results for $\| \cdot \|_{\oper}$. Let $A \in \RR^{d\otimes \ell}$ be a $\ell$-th order tensor. 
By the definition of the tensor spectral norm,  there exist  $u^{(1)}, \ldots, u^{(\ell) }\in \cS^{d-1}$ such that 
$
\| A \|_{\oper} = A(u^{(1)} , \ldots, u^{(\ell)} ) .
$
Moreover, for any $j\in [\ell]$ , there exists $v^{(j)} \in \cN_{\cS}(\epsilon) $ such that $ \| u^{(j)}   - v^{(j)} \|_2 \leq \epsilon$. 
In the following, we  prove tbound the difference between $A(u^{(1)} , \ldots, u^{(\ell)} ) $ and  $A(v^{(1) }, \ldots,  v^{(\ell)})$. To simplify the notation, we define  $\delta^{(j)} = u^{(j)} - v ^{(j)}$ by and let $\bar \delta^{(j)} = \delta^{(j)} / \| \delta^{(j)}\|_2$  for all $j\in [\ell]$.
By triangle inequality,   we have 
\#\label{eq:tensor_norm_immediate_term}
& \bigl | A(u^{(1)} , \ldots, u^{(\ell)} ) - A(v^{(1) }, \ldots,  v^{(\ell)}) \bigr | \notag \\
& \quad   \leq \bigl | A(u^{(1)} , \ldots, u^{(\ell)} ) - A(u^{(1) }, \ldots,  u^{(\ell-1)}, v^{(\ell)} ) \bigr |  \notag \\
&\quad\quad\quad   + \bigl |A(u^{(1) }, \ldots,  u^{(\ell-1)}, v^{(\ell)} )  - A(v^{(1) }, \ldots , v^{(\ell)} ) \bigr | .  
\#
For the first term on the right-hand side of  \eqref{eq:tensor_norm_immediate_term}, we have 
\$
& \bigl | A(u^{(1)} , \ldots, u^{(\ell)} ) - A(u^{(1) }, \ldots,  u^{(\ell-1)}, v^{(\ell)} ) \bigr |   =  \bigl |   A(u^{(1) }, \ldots,  u^{(\ell-1)}, \delta ^{(\ell)} ) \bigr | \notag \\
 &\quad   = \bigl |   A(u^{(1) }, \ldots,  u^{(\ell-1)}, \bar \delta ^{(\ell)} ) \bigr |  \cdot \| \delta ^{(\ell)}\|_2 \leq \| \delta^{(\ell)}  \|_2   \cdot \|A \|_{\oper},
\$
where the last inequality follows from the the definitions of $\| A \|_{\oper} $. Similarly, for the second term, we have 
\$
&  \bigl |A(u^{(1) }, \ldots,  u^{(\ell-1)}, v^{(\ell)} )  - A(v^{(1) }, \ldots , v^{(\ell)} ) \bigr | \notag \\
 &\quad  \leq  \bigl |A(u^{(1) }, \ldots,  u^{(\ell-1)}, v^{(\ell)} )  - A(u^{(1) }, \ldots , u^{(\ell-2)}, v^{(\ell-1)},  v^{(\ell)} ) \bigr | \notag \\
 &\quad \quad+ \bigl |  A(u ^{(1) }, \ldots , u^{(\ell-2)}, v^{(\ell-1)},  v^{(\ell)} )  - A(v^{(1) }, \ldots , v^{(\ell)} ) \bigr | \notag \\
 &\quad \leq\| \delta^{(\ell-1)}  \|_2  \cdot \| A \|_{\oper} +  \bigl |  A(u ^{(1) }, \ldots , u^{(\ell-2)}, v^{(\ell-1)},  v^{(\ell)} )  - A(v^{(1) }, \ldots , v^{(\ell)} ) \bigr |.
\$
Continuing the same argument, we finally obtain that 
\#\label{eq:some_meaningless1}
& \bigl | A(u^{(1)} , \ldots, u^{(\ell)} ) - A(v^{(1) }, \ldots,    v^{(\ell)} ) \bigr |  \notag \\
& \quad  \leq \bigl ( \| \delta ^{(1)} \|_2+ \ldots + \| \delta ^{(\ell)} \|_2 \bigr ) \cdot \| A \|_{\oper} \leq \ell \cdot   \epsilon \cdot  \| A \|_{\oper},
\#
where the last inequality follows from the definition of $\cN_{\cS}(\epsilon)$.   Thus,  by triangle inequality,
we have 
\$
 | A(v^{(1) }, \ldots,    v^{(\ell)} )|  
 &\geq \bigl | A(u^{(1)} , \ldots, u^{(\ell)} ) \bigr | - \bigl | A(u^{(1)} , \ldots, u^{(\ell)} ) - A(v^{(1) }, \ldots,    v^{(\ell)} ) \bigr | \notag\\ 
 &  \geq (1 - \ell \cdot \epsilon) \cdot \| A \|_{\oper},
\$
which proves \eqref{eq:tensor_net1}. Moreover, when $A$ is a symmetric tensor, we could choose $u^{(1)} = u^{(2)} = \ldots = u^{(\ell)} = u \in \cS^{d-1} $ and $ v^{(1)} = v^{(2)} = \ldots = v^{(\ell)} = v \in \cN_{\cS}(\epsilon)$ such that 
$
\| A \|_{\oper} = A(u, \ldots, u) 
$ 
and $\| u - v \|_2 \leq \epsilon$. Then by \eqref{eq:some_meaningless1} we have 
\$
| A (v, \ldots, v) | &  \geq \bigl  | A(u, \ldots, u) \bigr | - \bigl |  A(u, \ldots, u)  - A (v, \ldots, v) \bigr | \notag \\
& \geq \bigl (  1 - \ell\cdot \| u - v\|_2 \bigr ) \cdot \| A \|_{\oper} \geq (1 - \ell \cdot \epsilon) \cdot \| A \|_{\oper},
\$
which concludes \eqref{eq:tensor_net11}. Following the same argument with $\cS^{d-1}$ and $\cN_{\cS}(\epsilon)$ replaced by $\cS^{d-1}(r)$ and $\cN_{\cS}(\epsilon, r)$, respectively, we also have \eqref{eq:tensor_net2} and \eqref{eq:tensor_net21}. Thus, we conclude the proof.
\end{proof}

\subsection{Proof of Theorem \ref{thm:tensor_concentration}}\label{sec:proof_thm1}

In this section, we present the   proof of Theorem \ref{thm:tensor_concentration}.
  Since the results for $\hat M_1$ and $\hat M_2$    are established using similar techniques, in the following, we first present a detailed proof of the bound on     $\| \hat M_2 - \EE [ h_2(\mathcal{Z}) \cdot S_3(X)]\|_{\oper}$, then prove the other part by showing the differences.

\vspace{10pt}
{\noindent \bf Upper bound for  $\| \hat M_2 - \EE [ h_2(\mathcal{Z}) \cdot S_3(X)]\|_{\oper}$.}
First note that by the definition of third order score function $S_3$ in \eqref{eq:3rd_score},   we have
$
\hat M_2 = \hat T_1+ \hat T_2 ,
$
where $\hat M_2$ is given in \eqref{eq:moment_tensor2}, and we define $\hat T_1$ and $\hat T_2$ respectively  by
\#
  \hat T_1 & = \frac{1}{n} \sum_{i=1}^n Y^{(i)}  \cdot X^{(i)} \otimes X^{(i)} \otimes X^{(i)}, \label{eq:def_T1} \\
   \hat T_2 &  = \frac{1}{n} \sum_{i=1}^n \sum_{j=1}^d Y^{(i)}  \cdot \bigl  ( X^{(i)} \otimes e_j \otimes e_j +e_j \otimes X^{(i)}\otimes e_j + e_j \otimes e_j \otimes X^{(i)} \bigr ) \label{eq:def_T2},
\#
where    $Y^{(i) } = h_2( \cZ^{(i)} )  $ is the response of of the mixture of SIMs.
Note that we have
\$
\EE ( \hat M_2) = \EE[ h_2 (\cZ) \cdot    S_3(X) ]  = \sum_{j=1}^k \gamma _j^* \cdot \theta_j ^{* \otimes 3},
\$
where the second equality follows from Lemma \ref{lemma:moments}.   By the  triangle inequality, we have
\#\label{eq:use_triangle}
\bigl \| \hat M_2 - \EE (\hat M_2)  \bigr \|_{\oper}  \leq \bigl \| \hat T_1 - \EE (\hat T_1)  \bigr \|_{\oper} + \bigl \| \hat T_2- \EE (\hat T_2)  \bigr \|_{\oper}.
\#
In the sequel, we upper bound the two terms on the right-hand of \eqref{eq:use_triangle} separately.

 First, since $\hat T_1$ and $\hat T_2$ defined in \eqref{eq:def_T1} and \eqref{eq:def_T2} are symmetric tensors, the definition of tensor operator norm  in \eqref{eq:def_tensor_opernorm} yields that
\#\label{eq:T1_oper_norm_def}
\bigl \| \hat T_1 - \EE ( \hat T_1) \bigr \| _{\oper} =\sup_{u  \in \cS^{d-1}} \Bigl \{ \Bigl | \hat T_1 (u, u, u) - \EE \bigl [  \hat T_1(u,u,u) \bigr ]  \Bigr | \Bigr \},
\#
where $\cS^{d-1} = \{ u \in \RR^d \colon \| u \|_2 = 1 \}$ is the unit sphere in $\RR^d$. Similar to showing the concentration of random matrices, our derivation consists of two steps. In the first step, we  firstly apply the $\epsilon$-net argument  for tensor operator norm.  Then we bound  the concentration of $| \hat T_1(u, u, u) - \EE [ \hat T_1(u,u,u)] |$ for any fixed $u \in \cS^{d-1}$ and apply a  union bound over the $\epsilon$-net.

 To begin with, let $N_{\cS} (\epsilon )$ be the $\epsilon$-net of the unit sphere $\cS^{d-1}$ for any $\epsilon \in (0,1)$. We  now
 %
%
%
 appeal to Lemma~\ref{lem:net}, which shows that  taking  the supremum in \eqref{eq:T1_oper_norm_def} over  $N_{\cS} (1/6)$ instead of $\cS^{d-1}$ only incurs a small error. Specifically, we apply Lemma~\ref{lem:net} with $\ell = 3$, and $\epsilon = 1/6$ to $\hat T_1$ to obtain
\#\label{eq:use_epsilon_argument}
\bigl \| \hat T_1 - \EE ( \hat T_1) \bigr \| _{\oper}  \leq 2 \sup_{u  \in \cN_{\cS}(1/6) } \Bigl \{ \Bigl | \hat T_1 (u, u, u) - \EE \bigl [  \hat T_1(u,u,u) \bigr ]  \Bigr | \Bigr \}.
\#

In the next step, we derive a concentration inequality for  $\hat T_1(u,u,u)$ with fixed  $u \in \cS^{d-1}$, and apply a union bound of $\cN_{\cS}(1/6)$ to conclude the proof.

  By the definition of $\hat T_1$ in \eqref{eq:def_T1}, we have
 we have
\$
\hat T_1(u,u,u) = \frac{1}{n} \sum_{i=1}^n Y^{(i)} \cdot ( X^{(i) \top } u)^3.
\$
By Assumption \ref{assume:moments}, $\{ Y^{(i)} \}_{i\in [n]}$ are i.i.d.   sub-exponential random variables with $\psi_1$-norm bounded by $\Psi$. Moreover, since $\{X^{(i)} \}_{i\in [n]} $ are i.i.d. $ N(0, I_d)$ random vectors, $X^{(i) \top } u$ are independent standard Gaussian random variables for each $u \in \cS^{d-1}$. Thus $\| X^{(i)\top } u \|_{\psi_2} \leq \Upsilon_0$ where $\Upsilon_0 $ is a constant.
To bound $| \hat T_1(u,u,u) - \EE [  \hat T_1 (u,u,u) ] |$, we apply \eqref{eq:third_order_concentration} in Lemma \ref{lemma:sum_sub_exponential} to obtain that
\#\label{eq:apply_third_order}
&\PP \Bigl \{  \Bigl |  \hat T_1(u,u,u) - \EE  \bigl [ \hat T_1 (u, u, u) \bigr ] \Bigr | \geq t  / n\Bigr \} \notag \\
&\quad   \leq 4  \exp \Bigl (  - K_1 \cdot \min \big\{    [ t/ (\sqrt{n} \Upsilon_0^3 \cdot \Psi   )  ]^{2},    [ t / ( \Upsilon_0^3 \cdot \Psi )  ] ^{2/5} \bigr\}  \Bigr ),
\#
where $K_1$ is a constant. Moreover, this inequality holds for any $t \geq C_1 \cdot \Upsilon_0^3 \cdot \Psi$, where $C_1$ is a constant. Since both $\Psi$ and $\Upsilon_0$ are constants, we can rewrite \eqref{eq:apply_third_order} as
\#\label{eq:apply2_third_order}
& \PP \Bigl \{  \Bigl |  \hat T_1(u,u,u) - \EE  \bigl [ \hat T_1 (u, u, u) \bigr ] \Bigr | \geq t  / n\Bigr \} \notag \\
& \quad \leq 4  \exp  \bigl [  - \tilde C \cdot \min \bigl (     t^2 / n,     t  ^{2/5}  \bigr )   \bigr ]  ,
\#
where $\tilde C$ is a constant depending on $\Upsilon_0$ and $\Psi$.  Based on \eqref{eq:apply2_third_order}, we take a union bound over $\cN_{\cS}( 1/6)$.
As shown in    Lemma 5.2 in \cite{vershynin2010introduction}, the capacity of  $\cN_{\cS} (\epsilon) $ is bounded by $(1 + 2/ \epsilon)^{d}$ for any $\epsilon \geq 0$.  Thus we have  $| \cN_{\cS}(1/6) | \leq 13^d$, which implies that
\#\label{eq:union_1}
&\PP \biggl ( \sup_{u \in \cN_{\cS}(1/6)}  \Bigl \{   \Bigl |  \hat T_1(u,u,u) - \EE  \bigl [ \hat T_1 (u, u, u) \bigr ] \Bigr |  \Bigr \}  \geq t  / n \biggr ) \notag \\
&\quad  \leq 4  \exp  \bigl [  - \tilde C \cdot \min \bigl (     t^2 / n,     t  ^{2/5}  \bigr ) + \log 13 \cdot d   \bigr ]  ,
\#
Now we set $  \tilde C \cdot \min  (     t^2 / n,     t  ^{2/5}   ) = 6 d$. Note that $\tilde C$ is a constant. Solving the equation for $t$ yields~that
\#\label{eq:get_t}
t = K_{ 1} / 2 \cdot \max \bigl ( \sqrt{ nd}, d^{5/2} \bigr )
\#
for some constant $K_{ 1}$ depending on $\Psi$ and $\Upsilon_0$. Note that \eqref{eq:apply_third_order} holds for $t \geq C_1 \cdot \Upsilon_0^3 \cdot \Psi$. Thus, we require that 
$
 nd \geq ( 2 C_1 / K_1 )^2 \cdot  \Upsilon_0^6 \cdot  \Psi^2. 
$  In this case, $t$ defined in \eqref{eq:get_t} satisfies~\eqref{eq:union_1}.

Finally, combining \eqref{eq:use_epsilon_argument}, \eqref{eq:union_1}, \eqref{eq:get_t}, we conclude that
\#\label{eq:final_t1_bound}
\bigl \| \hat T_1 - \EE ( \hat T_1) \bigr \|_{\oper} \leq K_{ 1} \cdot \max \bigl ( \sqrt{d/n}, d^{5/2} / n \bigr )
\#
holds with probability at least $1 - 4 \exp( -3d)$.

Moreover, for $\hat T_2 $ defined \eqref{eq:def_T2} and any $u \in \cS^{d-1} $, we have
\$
\hat T_2 ( u, u, u) &= \frac{3}{n} \sum_{i=1}^n Y^{(i)} \cdot \sum_{j=1}^d (X^{(i)\top} u)  \cdot (u^\top e_j)^2 \\ &  =  \frac{3}{n} \sum_{i=1}^n Y^{(i)} \cdot (X^{(i)\top}u  )
 \cdot \sum_{j=1}^d  u_j^2    = \frac{3}{n} \sum_{i=1}^n Y^{(i)} \cdot (X^{(i)\top}u  ),
  \$
  where the last equality holds since $\| u \|_2 = 1$.
Note that  $Y^{(i)}$ and $(X^{(i)\top}u  )$ are sub-exponential  and sub-Gaussian random variables, respectively, with $\|X^{(i)\top}u\|_{\psi_2} \leq \Upsilon_0$ and $\| Y \|_{\psi_1} \leq \Psi$.  By  \eqref{eq:third_order_concentration2} in Lemma \ref{lemma:sum_sub_exponential}, we have
\#\label{eq:apply_third_order2}
& \PP \Bigl \{  \Bigl |  \hat T_2(u,u,u) - \EE  \bigl [ \hat T_2 (u, u, u) \bigr ] \Bigr | \geq t  / n\Bigr \} \notag \\
& \quad \leq  4  \exp \Bigl (  - K_2 \cdot \min \big\{    [ t/ (\sqrt{n} \Upsilon_0  \cdot \Psi   )  ]^{2},    [ t / ( \Upsilon_0  \cdot \Psi )  ] ^{2/3} \bigr\}  \Bigr ),
\#
where $K_2 >0 $ is a constant.  Note that \eqref{eq:apply_third_order2} holds for any $t\geq C_1 \cdot \Upsilon_0   \Psi$. Then we take a union bound for all $u \in \cN_{\cS} (1/6)$ in \eqref{eq:apply_third_order2} to obtain that
\#\label{eq:union_2}
& \PP \biggl ( \sup_{u \in \cN_{\cS}(1/6)}  \Bigl \{   \Bigl |  \hat T_2(u,u,u) - \EE  \bigl [ \hat T_2 (u, u, u) \bigr ] \Bigr |  \Bigr \}   \geq t  / n \biggr ) \notag 
\\
& \quad \leq 4  \exp  \bigl [  - \tilde C \cdot \min \bigl (     t^2 / n,     t  ^{2/3}  \bigr ) + \log 13 \cdot d   \bigr ]  ,
\#
where $\tilde C$ is a constant depending on $\Upsilon_0$ and $\Psi$. Similar to \eqref{eq:get_t}, setting $\tilde C \cdot \min  (     t^2 / n,     t  ^{2/3}   ) = 6 d$ implies that
\#\label{eq:get_t2}
t = K_{2} / 2 \cdot \max \bigl ( \sqrt{ nd}, d^{3/2} \bigr )
\#
for some constant $K_{ 2}$. When $nd \geq ( 2C_1 / K_2 )^2 \cdot \Upsilon_0 ^2 \cdot   \Psi^2 $, $t$ defined in \eqref{eq:get_t2} satisfies that 
$t \geq C_1 \cdot \Upsilon_0 \Psi$, which implies that \eqref{eq:union_2} holds for such a $t$.
Finally, combining \eqref{eq:union_2} and \eqref{eq:get_t2}, we obtain
\#\label{eq:final_t2_bound}
\bigl \| \hat T_2 - \EE ( \hat T_2) \bigr \|_{\oper} \leq K_{2} \cdot \max \bigl ( \sqrt{d/n}, d^{3/2} / n \bigr )
\#
with probability at least $1 - 4 \exp( -3d)$. Then   combining \eqref{eq:final_t1_bound} and \eqref{eq:final_t2_bound},  since $d^{5/2}  > d^{3/2}$, we conclude that
\#\label{eq:final_spectral_bound}
\bigl \| \hat M_2 - \EE[ h_2 (\cZ) \cdot    S_3(X) ] \bigr \|_{\oper}  \leq K \cdot \max \bigl ( \sqrt{d/n}, d^{5/2} / n \bigr )
\#
holds with   with probability at least $1 -8 \exp(-3d)$, where the constant $K$ in \eqref{eq:final_spectral_bound} can be chosen to be $K_1 + K _2$.
  When $d $ is sufficiently large such that $\exp(d) \geq 8$, we conclude that \eqref{eq:final_spectral_bound} holds with probability at least $1 - \exp(-2d)$. Recall that we require $n d \geq C \cdot \max \{ \Upsilon_1^6, \Upsilon_1^2 \} \cdot \Psi^2 $ for some constant $C >0$, which holds when $n$ and $d$ are sufficiently large.

\vspace{10pt}
{\noindent \bf Upper bound for  $\| \hat M_1 - \EE [ h_1(\mathcal{Z}) \cdot S_3(X)]\|_{\oper}$.}
 To conclude the proof, it remains to bound  $\| \hat M_1 - \EE [ h_1(\mathcal{Z}) \cdot S_3(X)]\|_{\oper}$.  For notational simplicity,   we denote  $ h_1(\cZ^{(i)} ) = k^{-1} \sum_{j\in [k]} Z_j^{(i)} $ by $W^{(i) } $ for any $i \in [n]$.  Then by \eqref{eq:moment_tensor1},  we can write  $\hat M_1 = \sum_{i =1}^n W^{(i) } \cdot S_3(X^{(i)}) = \hat T_3  +\hat T_4$, where we define $\hat T_3$ and $\hat T_4$ respectively~by
\$
& \hat T_3 = \frac{1}{n} \sum_{i=1}^n W^{(i)}  \cdot X^{(i)} \otimes X^{(i)} \otimes X^{(i)},  \\
& \hat T_4 = \frac{1}{n} \sum_{i=1}^n \sum_{j=1}^d W^{(i)}  \cdot \bigl  ( X^{(i)} \otimes e_j \otimes e_j +e_j \otimes X^{(i)}\otimes e_j + e_j \otimes e_j \otimes X^{(i)} \bigr ) .
\$
We note that here the $\hat T_3$ and $\hat T_4$ is defined in the same fashion as $\hat T_1$ and $\hat T_2$ defined in   \eqref{eq:def_T1} and \eqref{eq:def_T2} with $\{Y^{(i)} \}_{i\in [n]}$ replaced by $\{W^{(i)} \}_{i\in [n]}$.
Then triangle inequality implies that
\#\label{eq:use_triangle2}
\bigl \| \hat M_1 -  \EE [ h_1(\mathcal{Z}) \cdot S_3(X)]  \bigr \|_{\oper}  \leq \bigl \| \hat T_3 - \EE (\hat T_3)  \bigr \|_{\oper} + \bigl \| \hat T_4- \EE (\hat T_4)  \bigr \|_{\oper}.
\#
Moreover, by Assumption \ref{assume:moments}, we have $\| Z_j \| _{\psi_1} \leq \Psi$ for each $j\in [k]$.  We   appeal to  Lemma \ref{lemma:psi_1_norm} to obtain  that $\| h_1(\cZ) \|_{\psi_1} \leq \Psi$, which implies that
$\| W^{(i) } \| _{\psi_1} \leq \Psi$ for any $i \in [n]$.
Thus, following the same derivation for \eqref{eq:final_spectral_bound} with $Y^{(i)}$ replaced by $W^{(i)}$ for all $i\in [n]$, we obtain that
\#\label{eq:final_spectral_bound2}
\bigl \| \hat M_1 - \EE[ h_1 (\cZ) \cdot    S_3(X) ] \bigr \|_{\oper}  \leq K  \cdot \max \bigl ( \sqrt{d/n}, d^{5/2} / n \bigr ).
\#
with probability at least $1- \exp(-2d)$. Here the   constant  $K $ in \eqref{eq:final_spectral_bound2} can be set as  $K_1 + K_2$, where $K_1$ and  $K_2$ are given in \eqref{eq:get_t} and \eqref{eq:get_t2}, respectively.
Finally, combining \eqref{eq:final_spectral_bound} and \eqref{eq:final_spectral_bound2}, we conclude the proof of Theorem \ref{thm:tensor_concentration}.


\subsection{Proof of Theorem \ref{thm:tensor_concentration1}}\label{sec:proof_thm2}


The proof of Theorem \ref{thm:tensor_concentration1}    is similar to that of Theorem~\ref{thm:tensor_concentration}. 
Recall that we define $W^{(i)} = k^{-1} \sum_{j\in [k]} Z_j^{(i)}$, whose $\psi_1$-norm is bounded by $\Psi$.  Using the similar argument as in Theorem~\ref{thm:tensor_concentration}, we only need to consider $\hat M_2 $; the result for $\hat M_1$ follows similarly by replacing $Y^{(i)}$ by $W^{(i)}$ for all $i\in [n]$. 

Moreover, recall that we define  $\hat T_1$ and $\hat T_2$ in \eqref{eq:def_T1} and \eqref{eq:def_T2}, respectively, which satisfy $\hat M_2= \hat T_1 + \hat T_2$.
 For any $r \in [d]$, by the definition of $\| \cdot \|_{\oper, r}$ in \eqref{eq:def_sparse_tensor_opernorm} and the triangle inequality, we have  
\#\label{eq:use_triangle2}
\bigl \| \hat T - \EE (\hat T)  \bigr \|_{\oper, r}  \leq \bigl \| \hat T_1 - \EE (\hat T_1)  \bigr \|_{\oper,r } + \bigl \| \hat T_2- \EE (\hat T_2)  \bigr \|_{\oper,r }.
\#
In the sequel, we bound the two terms on the right-hand side of \eqref{eq:use_triangle2} separately. 

We first bound $ \| \hat T_1 - \EE (\hat T_1)   \|_{\oper,r }$.   
Let $\cN_{\cS}(\epsilon, r)$ be the $\epsilon$-net of $\cS^{d-1}(r) = \{ u \in \cS^{d-1} \colon \| u \|_0 \leq r\}$. 
We apply the $\epsilon$-net argument for $\| \cdot \|_{\oper, r}$. By Lemma~\ref{lemma:tensor_net} with $\ell = 3$ and $\epsilon= 1/6$, we have 
\#\label{eq:use_epsilon_argument2}
\bigl \| \hat T_1 - \EE ( \hat T_1) \bigr \| _{\oper,r }  \leq 2 \sup_{u  \in \cN_{\cS}(1/6,r ) } \Bigl \{ \Bigl | \hat T_1 (u, u, u) - \EE \bigl [  \hat T_1(u,u,u) \bigr ]  \Bigr | \Bigr \}.
\#
Note that for any $u \in \cS^{d-1}$, \eqref{eq:apply2_third_order} gives an upper bound of the tail probability 
\$
\PP \Bigl \{  \Bigl |  \hat T_1(u,u,u) - \EE  \bigl [ \hat T_1 (u, u, u) \bigr ] \Bigr | \geq t  / n\Bigr \}.
\$
Based on \eqref{eq:apply2_third_order} , we take a union bound over 
$\cN_{\cS}( 1/6,r)$. The cardinality of $\cN_{\cS}(1/6, r)$ satisfies 
\#\label{eq:cardinarlity_sparse_net}
\bigl | \cN_{\cS}(1/6, r) \bigr |  \leq {d \choose r} \cdot 13^s \leq (13 ed /r)^{r}. 
\#
Combining \eqref{eq:apply2_third_order} and \eqref{eq:cardinarlity_sparse_net}, we have
\#\label{eq:union_11}
&\PP \biggl ( \sup_{u \in \cN_{\cS}(1/6, r)}  \Bigl \{   \Bigl |  \hat T_1(u,u,u) - \EE  \bigl [ \hat T_1 (u, u, u) \bigr ] \Bigr |  \Bigr \}\geq t  / n \biggr ) \notag \\
& \quad\qquad\qquad \leq 4  \exp  \bigl [  - \tilde C \cdot \min \bigl (     t^2 / n,     t  ^{2/5}  \bigr ) + r \cdot  \log (13 ed /r)   \bigr ]  ,
\#
where $\tilde C$ is a constant.
Setting  $  \tilde C \cdot \min  (     t^2 / n,     t  ^{2/5}   ) = 7r \cdot \log (d/r)$. Note that $\tilde C$ is a constant. Solving the equation for $t$ yields~that 
\#\label{eq:get_t12}
t = K_{3} / 2 \cdot \max \bigl \{ \sqrt{ r n\cdot \log (d/r) }, [ r\cdot \log (d/r)] ^{5/2} \bigr \}
\# 
for some constant $K_{3}$ depending on $\Psi$ and $\Upsilon_0$.
Finally, combining \eqref{eq:use_epsilon_argument2}, \eqref{eq:union_11}, \eqref{eq:get_t12}, we conclude that 
\#\label{eq:final_t1_bound2}
\bigl \| \hat T_1 - \EE ( \hat T_1) \bigr \|_{\oper, r } \leq K_{3} \cdot \max \biggl \{ \biggl [ \frac{r \log (d/r )}{n} \biggr ]^{1/2}, \frac{ [  r \log (d/r) ]^{5/2}} {n} \biggr \} 
\#
with probability at least $1 - 4 \exp[  -3r \cdot \log(d/r) ]$.

It remains to bound  $\| \hat T_2  - \EE ( \hat T_2) \|_{\oper, r}$. Following the similar argument, by taking an union bound over $\cN_{\cS}(\epsilon, r)$  using the concentration inequality in \eqref{eq:apply_third_order}, we have
\#\label{eq:union_22}
& \PP \biggl ( \sup_{u \in \cN_{\cS}(1/6)}  \Bigl \{   \Bigl |  \hat T_2(u,u,u) - \EE  \bigl [ \hat T_2 (u, u, u) \bigr ] \Bigr |  \Bigr \}\geq t  / n \biggr ) \notag \\
&\qquad \qquad\qquad\leq 4  \exp  \bigl [  - \tilde C \cdot \min \bigl (     t^2 / n,     t  ^{2/3}  \bigr ) +  r \cdot  \log (13 ed /r)    \bigr ]  ,
\#
Now we set $\tilde C \cdot \min  (     t^2 / n,     t  ^{2/3}   ) = 7r \cdot \log (d/r)$ in \eqref{eq:union_22},  which implies that 
\#\label{eq:get_t22}
t = K_{4} / 2 \cdot \max \bigl ( \sqrt{ nd}, d^{3/2} \bigr ) 
\#
for some constant $K_{4}$. Finally, combining \eqref{eq:union_22} and \eqref{eq:get_t22}, we obtain that 
\#\label{eq:final_t2_bound2}
\bigl \| \hat T_2 - \EE ( \hat T_2) \bigr \|_{\oper, r} \leq K_{4} \cdot  \max \biggl \{ \biggl [ \frac{r \log (d/r )}{n} \biggr ]^{1/2}, \frac{ [  r \log (d/r) ]^{3/2}} {n} \biggr \} 
\#
with probability at least $1 - 4\exp[  -3r \cdot \log(d/r) ]$.

Moreover, note that  $r \log (d/r)  \leq \log d$ and  $r \log d \geq 1$.
Therefore, combining \eqref{eq:final_t1_bound2}  and \eqref{eq:final_t2_bound2}, we conclude that, with probability at least 
$1 - 8\exp[  -3r \cdot \log d]$, we have 
\#\label{eq:M2_bound_sparse_final}
 \bigl \| \hat M_2 - \EE[ h_1 (\cZ) \cdot    S_3(X) ] \bigr \|_{\oper, r}  &\leq   K \cdot \max \biggl \{ \biggl ( \frac{r \log d }{n} \biggr )^{1/2}, \frac{ ( r \log d  )^{5/2}} {n} \biggr \}  ,
\#
where $K  = K_3 + K_4$.
Thus we conclude the proof for $\hat M_2$. Finally, we recall that the bound on $\| \hat M_1 - \EE[ h_1 (\cZ) \cdot    S_3(X) ]  \|_{\oper, r}$ can be derived in the similar  fashion as \eqref{eq:M2_bound_sparse_final} by replacing $Y^{(i)}$ by $W^{(i)}$ for each $i \in [n]$. Therefore, we conclude the proof of  Theorem \ref{thm:tensor_concentration1}.


\bibliographystyle{plainnat}
\bibliography{overcompnn}

\end{document}